\documentclass[a4paper,10pt]{article}
\usepackage[utf8]{inputenc}
\usepackage{lmodern} 
\usepackage[T1]{fontenc}
\usepackage{textcomp}
\usepackage[scr=boondoxo,frak=boondox,bb=boondox]{mathalfa}
\title{\normalsize{\textbf{LINKS OF SINGULARITIES OF INNER NON-DEGENERATE MIXED FUNCTIONS}}}
\author{Raimundo N. Araújo dos Santos  \hspace{0.1cm} Benjamin Bode \and Eder L. Sanchez Quiceno}
\newcommand{\Addresses}{{
  \bigskip
  \footnotesize

R.~N.~Araújo dos Santos, \textsc{Institute of Mathematics and Computer Science (ICMC)\\
University of São Paulo (USP), Avenida Trabalhador São-Carlense, 400 - Centro\\
CEP: 13566-590 - São Carlos - SP, Brazil}\par\nopagebreak
  \textit{E-mail address}: \texttt{rnonato@icmc.usp.br}

  \medskip

B.~Bode, \textsc{Instituto de Ciencias Matemáticas (ICMAT), Consejo Superior de Investigaciones Científicas (CSIC), Campus Cantoblanco UAM,\\
C/ Nicolás Cabrera, 13-15, 28049 Madrid, Spain}\par\nopagebreak
  \textit{E-mail address}: \texttt{benjamin.bode@icmat.es}

  \medskip

E.~L.~Sanchez Quiceno, \textsc{Institute of Mathematics and Computer Science (ICMC)\\
University of São Paulo (USP), Avenida Trabalhador São-Carlense, 400 - Centro\\
CEP: 13566-590 - São Carlos - SP, Brazil}\par\nopagebreak
  \textit{E-mail address}: \texttt{ederleansanchez@usp.br}

}}
\date{}
\usepackage{graphicx}
\usepackage[all]{xy}
\usepackage{mathrsfs}
\usepackage[usenames,dvipsnames]{color}
\usepackage{makeidx}
\usepackage{subcaption}
\usepackage{graphicx}
\usepackage{float}
\usepackage{indentfirst}
\usepackage{amsmath}
\usepackage{amsthm}
\usepackage{graphicx,color}
\usepackage{setspace}
\usepackage{epigraph}
\usepackage{lmodern}
\usepackage{amsfonts,amssymb,amsmath,latexsym}
\usepackage{verbatim}
\usepackage{lscape}
\usepackage{tikz-cd}
\usetikzlibrary{calc,intersections,through,backgrounds}
\usepackage{xfrac}    
\usepackage{pgfplots}
\usepackage{mathtools}
\usepackage{tikz}
\usepackage{faktor}
\usepackage{amsmath}
\usepackage{braids}
\usepackage{pinlabel}
\usepackage{titlesec}
\usepackage{lipsum}
\usepackage{soul}
\titleformat{\section}
        [block]
        {\small \mdseries\centering}
        {\thesection}
        {1ex}
        {}
        [
        ]%
        \titleformat{\subsection} [block] {\large \mdseries}
        {\thesubsection} {1ex} {}
        [
        ] 
\usepackage[english]{babel}
\usepackage{hyperref}
\hypersetup{
    colorlinks=true,
    linkcolor=blue,
    filecolor=magenta,      
    urlcolor=cyan,
}
\makeindex
\newtheorem{teo}{Theorem}[section]
\newtheorem{corolario}[teo]{Corollary}
\newtheorem{prop}[teo]{Proposition}
\newtheorem{lemma}[teo]{Lemma}
\newtheorem{definition}[teo]{Definition}
\newtheorem{ex}[teo]{Example}
\newtheorem{obs}[teo]{Remark}

\newcommand{\C}{\mathbb{C}}       
\newcommand{\R}{\mathbb{R}}       
\newcommand{\N}{\mathbb{N}}       

\newcommand{\rme}{\mathrm{e}}
\newcommand{\rmi}{\mathrm{i}}

\newcommand{\defeq}{\mathrel{\mathop:}=}
   \makeatletter
\@namedef{subjclassname@2020}{%
  \textup{2020} Mathematics Subject Classification}
\begin{document}
\maketitle
\begin{abstract}
We introduce the notion of a (strongly) inner non-degenerate mixed function $f:\mathbb{C}^2\to\mathbb{C}$. We show that inner non-degenerate mixed polynomials have weakly isolated singularities and strongly inner non-degenerate mixed polynomials have isolated singularities. Furthermore, under one additional assumption, which we call ``niceness'', the links of these singularities can be completely characterized in terms of the Newton boundary of $f$. In particular, adding terms above the Newton boundary does not affect the topology of the link. This allows us to define an infinite family of real algebraic links in the 3-sphere.
    \vspace{0.2cm} 

\noindent \textit{Keywords:} real algebraic links, mixed singularities, semiholomorphic function, isolated singularity, Newton non-degenerate, P-fibered braid, Newton boundary.  
   
     \vspace{0.2cm} 
     
\noindent \textit{Mathematics Subject Classification:} 
Primary 32S55, 57K10; Secondary 14M25, 14P05, 14P15, 14P25, 32S05. 	
   
\end{abstract}
\section{INTRODUCTION}
It is a key objective in singularity theory to obtain information about geometric, topological or analytical properties of a (algebraic) variety in a neighbourhood of a singular point by studying the defining set of (polynomial) equations. A classical example of this occurs in the study of complex plane curves, the zeros of a complex polynomial $f:\mathbb{C}^2\to\mathbb{C}$, where we can obtain the Puiseux expansion of each branch via the Newton polygon of $f.$ See for instance \cite{brieskorn}. The Puiseux pairs then allow a complete characterization of algebraic curves and their singularities from a topological viewpoint (cf. for example  \cite{eisenbud}). Analogous results does not hold  in general over fields that are not algebraically closed, such as the real numbers, since there might not even exists a Puiseux expansion for every branch of the curve. In this paper we show that under some assumptions we can nonetheless completely describe the topology of a real algebraic variety near a singularity by studying the boundary of an appropriately defined Newton polygon.

We consider polynomial maps $f:\mathbb{R}^4\to\mathbb{R}^2$ and their singularities, i.e., points where their Jacobian matrix does not have full rank. We say that the origin $O$ is a \textit{weakly isolated singularity} if it is a singularity, i.e. $O\in \Sigma_{f}\cap V_f,$ and there is an open neighbourhood $U\subset\mathbb{R}^4$ of the origin such that $U\cap V_f\cap\Sigma_f=\{O\}$, where $V_f:=\{f=0\}$ is the variety defined by $f$ and $\Sigma_f:=\{x\in U: \text{rank J}f(x)\leq 1\}$ is the set of critical points (or, the singular locus) of $f.$ 

We say that $O$ is an \textit{isolated singularity} if it is a singularity and there is a neighbourhood $U$ of the origin such that $U\cap\Sigma_f=\{O\}.$

Naturally, holomorphic polynomials functions form a special case of this family of real polynomial mappings and indeed this real setting has several similarities with the complex case. Intersecting $V_f$ with a 3-sphere of a small enough radius centered at the singular point, it results in a smooth link whose isotopy type is independent of the chosen radius. In contrast to the local topological classification of the singularities of complex plane curves, the links that can arise as the links of isolated singularities of real polynomial maps, called \textit{real algebraic links}, are not classified yet. 

Milnor proved in \cite{Milnor1968} that the set of real algebraic links is a subset of the set of fibered links and it was conjectured by Benedetti-Shiota in \cite{Benedetti_Shiota1998} that these two sets are identical. In general, it is well-known that every link arises as the link of a weakly isolated singularities as proved by Akbulut and King in \cite{Akbulut_King1981}. However, besides those links that come from holomorphic polynomials from $\C^2\to\C$ the current literature only offers relatively few examples of fibered links that are known to admit realizations as real algebraic links \cite{Bode2019, bodesat, bodebraid, Looijenga1971, Perron1982, Pichon2005, Rudolph1987}. As a consequence of our main results we can add an infinite family to this list of examples.


Following Oka \cite{Oka2010} we write real polynomials $f:\mathbb{R}^4\to\mathbb{R}^2$ as \textit{mixed functions}, that is, as polynomials $f:\mathbb{C}^2\to\mathbb{C}$ in complex variables $u$ and $v$ and their complex conjugates $\overline{u}$ and $\overline{v}$. Mixed functions are a natural generalization of holomorphic functions as will be shown in the Section~\ref{sec:prem}. Working with mixed functions allows us to find analogues of tools that have proven useful in the complex setting. In \cite{Oka2010} Oka also defines the \textit{Newton polygon} $\Gamma_+(f)$ for any such mixed polynomial function, which in the case that it does not depend on the variables $\overline{u}$ and $\overline{v}$ naturally reduces to the usual definition of the Newton polygon of a complex plane curve. Using this definition he finds conditions on the boundary of the Newton polygon, called \textit{(strong) non-degeneracy} and \textit{convenience} that, if satisfied simultaneously, guarantee that the mixed function has a weakly isolated singularity (or, an isolated singularity) at the origin. 

Similarly, in the case of complex plane curves Mondal's and Wall's conditions of being partially non-degenerate \cite{mondal} or inner non-degenerate \cite{wall} (which in this case are equivalent) imply an isolated singularity \cite{stevens}. Although these definitions have been generalized to fields with non-zero characteristic \cite{boubakri}, the concepts and results in general do not translate to fields that are not algebraically closed.

Motivated by the holomorphic setting we introduce in this paper what it means for a mixed polynomial to be (strongly) inner non-degenerate. For holomorphic functions our definition of being inner non-degenerate matches with those of \cite{mondal} and \cite{wall}. The set of (strongly) inner non-degenerate mixed polynomials contains all convenient (strongly) non-degenerate polynomials and we show that inner non-degenerate mixed functions have weakly isolated singularities and strongly inner non-degenerate mixed functions have isolated singularities, even if they are not necessarily convenient. Our main result is that under our new inner non-degeneracy condition and some small additional assumption (``being \textit{nice}'') the links of the resulting singularities are completely determined by the terms on the boundary of the Newton polygon, as follows.
\begin{teo}
\label{thm:main}
Let $f:\mathbb{C}^2\to\mathbb{C}$ be a mixed polynomial with a nice, inner non-degenerate boundary. Then the link of the singularity is ambient isotopic to $L([L_1,L_2,\ldots,L_{N-1}],L_N)$.
\end{teo}

Here $N$ is the number of compact 1-faces of the Newton boundary. Throughout the paper we assume that $N\geq 1$, which implies that all mixed polynomials in this paper satisfy $f(O)=0$. Every $L_i$, $i=1,2,\ldots,N$, is a link (i.e., a finite disjoint union of smoothly embedded circles) that corresponds to the zeros of an appropriate function $g_i$ associated with the $i$th 1-face, where $L_1$ is a link in $\C\times S^1$, $L_i$ with $i\in\{2,3,\ldots,N-1\}$ is a link in $(\C\backslash\{0\})\times S^1$ and $L_N$ is a link in $S^1\times \C$. The link $L([L_1,L_2,\ldots,L_{N-1}],L_N)$ can then be explicitly constructed (see Definitions~\ref{def:russiandoll2} and \ref{def:russiandoll3}, Figure~\ref{fig:links}) as a link in $S^3$ and contains the $L_{i}$s (interpreted as links in $S^3$ by appropriately embedding $\C\times S^1$ and $S^1\times\C$ into $S^3$) as sublinks. The construction in Definitions~\ref{def:russiandoll2} and \ref{def:russiandoll3} is formulated in terms of parametrisations of the $L_i$s, but the isotopy class of $L([L_1,L_2,\ldots,L_{N-1}],L_N)$ only depends on the isotopy classes of the $L_i$s in $\C\times S^1$, $(\C\backslash\{0\})\times S^1$ and $S^1\times \C$, respectively.

It follows from the classification of holomorphic functions $\C^2\to\C$ with an isolated singularity (see, \cite{King1978, Saeki1989}) and by Fukui and Yoshinaga in \cite{FukuiYoshinaga1985} that the links of the singularities of non-degenerate, holomorphic polynomials from $\C^2\to\C$ with the same Newton principal part $f_{\Gamma}$, consisting of all the terms on the Newton boundary $\Gamma$, are ambient isotopic. Since the isotopy class of $L([L_1,L_2,\ldots,L_{N-1}],L_N)$ only depends on the Newton principal part of a mixed polynomial, our result can be seen as a generalization of the result in the holomorphic case.


A special family of mixed polynomials is the set of \textit{semiholomorphic} polynomials, which by definition do not depend on the variable $\overline{u}$ and are therefore holomorphic in the variable $u$. This makes the family of semiholomorphic polynomials an interesting intermediate between the general mixed setting and the holomorphic setting. They admit various connections to the study of complex polynomials and (from a topological perspective) braided vanishing sets. We show for example that the links $L_i$, $i=1,2,\ldots,N$, in Theorem \ref{thm:main} can be understood as closed braids in $\C\times S^1$ if the Newton principal part of $f$ is semiholomorphic. We write $B_i$ for the corresponding braid in $\C\times [0,2\pi]$. Furthermore, the holomorphicity in $u$ makes it significantly easier to find critical points and in turn prove that a singularity is isolated. In fact, several constructions of real polynomials with isolated singularities have taken advantage of these desirable properties, as one can see for instance in \cite{Bode2019, Rudolph1987}.

Taking advantage of that, we prove that for a mixed polynomial $f$ whose Newton principal part $f_{\Gamma}$ is semiholomorphic the conditions from Theorem~\ref{thm:main} can be further weakened. More precisely, we prove:
\begin{prop}\label{prop:main} 
Let $f:\mathbb{C}^2\to\mathbb{C}$ be an inner non-degenerate mixed polynomial, whose Newton principal part $f_{\Gamma}$ is semiholomorphic. Then $f$ has a weakly isolated singularity at the origin and the link $L_f=L_{f_{\Gamma}}$ is isotopic to the closure of the braid $B(B_1,B_2,\ldots,B_{N})$ if $f$ is $u$-convenient. If $f$ is not $u$-convenient, then the link of the singularity is $B^o(B_1,B_2,\ldots,B_N)$. 
\end{prop}
The property of $u$-convenience is defined in Section~\ref{sec:prem}. The explicit constructions of the braids $B(B_1,B_2,\ldots,B_{N})$ and the links $B^o(B_1,B_2,\ldots,B_N)$ are given in Definition~\ref{def:russiandoll} and Figure~\ref{fig:braids}. As mentioned above the braids $B_i$, $i=1,2,\ldots,N$, are the links $L_i$, $i=1,2,\ldots,N$, from Theorem~\ref{thm:main} interpreted as closed braids. Thus, $B_1$ is a braid in $\C\times[0,2\pi]$ and $B_i$ with $i>1$ is a braid in $(\C\backslash\{0\})\times[0,2\pi]$, defined as the zeros of an appropriate function $g_i$ associated with the $i$th 1-face. Again, the construction of the braids $B(B_1,B_2,\ldots,B_{N})$ and the links $B^o(B_1,B_2,\ldots,B_N)$ is given in terms of parametrizations of curves. However the isotopy class of the constructed link only depends on the braid isotopy classes of the $B_i$s in $\C\times[0,2\pi]$ and $(\C\backslash\{0\})\times[0,2\pi]$, respectively. Note that in Proposition~\ref{prop:main} the assumption that $f$ is nice is no longer needed.

Combining our results on strong inner non-degeneracy with the theory of $P$-fibered braids \cite{bodesat}, we obtain a {\it new family of real algebraic links} as the below result shows.
\begin{teo}\label{cor:main}
Let $(B_1,B_2,\ldots,B_N)$ be a sequence of braids such that for all $i\in\{1,2,\ldots,N\}$ the braid $B_i$ consists of $s_i$ strands and is P-fibered with $O$-multiplicity $\sum_{j<i}s_j$ . Let $r_N:=1$.\\ Then for every sequence of positive integers $(r_j,r_{j+1},\ldots,r_{N})$ there is a positive integer $N_{j-1}$ such that the closure of $B(B_1^{2r_1},B_2^{2r_2},\ldots,B_{N-1}^{2r_{N-1}},B_N^{2r_N})$ is real algebraic if $r_j>N_j$ for all $j=1,2,\ldots,N-1$.
\end{teo}

Here $B_i^{2r_i}$ refers to $2r_i$th power of $B_i$ in the braid group.

This paper focuses on classical links, that is, embeddings of 1-dimensional closed manifolds in 3-dimensional spaces, which means that we only consider mixed polynomials $f:\C^2\to\C$. However, we expect that our definitions and most of our results stated in the paper have analogues in the more general setting of mixed polynomials $f:\C^n\to\C$.

\vspace{0.3cm}
The remainder of this paper is structured as follows.
In Section~\ref{sec:prem} we recall the concepts of mixed polynomials and mixed singularities as well as the construction of the Newton polygon and different notions of non-degeneracy. We also review definitions and properties of braids and their relation to polynomials and singularities. In Section~\ref{section3} we introduce our notion of being inner non-degenerate, investigate its relation to other notions of non-degeneracy and show that it implies the existence of a weakly isolated singularity at the origin. The braids $(B_1,B_2,\ldots,B_N)$ and $B^o(B_1,B_2,\ldots,B_N)$ are described in Section~\ref{section4}, where we also prove Proposition~\ref{prop:main}. In Section~\ref{section5}, we define the links $L([L_1,L_2,\ldots,L_{N-1}],L_N)$ that appear as the links of singularities of mixed polynomials with a nice, inner non-degenerate boundary. We also prove the main result Theorem~\ref{thm:main}. In Section~\ref{section6} we study the isolatedness of singularities of mixed polynomials with regards to our definition of being strongly inner non-degenerate. In particular, we show that strongly inner non-degenerate polynomials have isolated singularities and show that convenient strongly non-degenerate polynomials are strongly inner non-degenerate. We use this property to prove Theorem~\ref{cor:main} in Section~\ref{section7}.

\vspace{0.3cm}
{\bf Acknowledgments:} R. Araújo dos Santos thanks Fapesp/Tematic project process: 2019/21181-0. B. Bode acknowledges funding from the European Union’s Horizon 2020 Research and Innovation Programme under the Marie Sklodowska-Curie grant agreement No 101023017, and E. Sanchez Quiceno acknowledges the
 supports by grants 2019/11415-3 and 2017/25902-8, São Paulo Research Foundation (FAPESP). The authors are thankful to Professor Osamu Saeki from Kyushu University, Japan, for all suggestions and comments that contributed to enrich the paper.

\section{PRELIMINARIES}\label{sec:prem}
\subsection{Mixed polynomials and Newton boundaries}
In this section we review the essential tools concerning mixed functions and their Newton boundaries. Most of this background material can be found in more detail in \cite{Oka2010}.

We consider the germ of a mixed polynomial $f:(\C^2,0)\to (\C,0)$, $$f(z,\bar{z})=\sum_{\nu,\mu} c_{\nu,\mu}z^{\nu}\bar{z}^{\mu},$$
where $z = (u,v)$, $ \bar{z}= (\bar{u},\bar{v})$, $z^{\nu} = {u}^{\nu_{1}}v^{\nu_{2}}$ for $\nu = (\nu_{1} ,\nu_{2})$ (respectively $\bar{z}^{\mu}=\bar{u}^{\mu_{1}}{\bar{v}}^{\mu_{2}}$ for $\mu=(\mu_{1},\mu_{2})$). Most readers might be more accustomed to the notations $z=(z_1,z_2,\ldots,z_n)$ and $\bar{z}=(\bar{z}_1,\bar{z}_2,\ldots,\bar{z}_n)$. However, in this case where $n=2$ we prefer to work with complex variables $u$ and $v$ in order to be consistent with the notation in \cite{Bode2019}. 

We will slightly abuse notation and use the same notation for the germ of a function or a set and its representative. Furthermore, the origins in $\C^2$ and in $\C$ will both be denoted by 0 if there is no risk of confusion. Occasionally, we will also write $O$ for the origin in $\C^2$.


\vspace{0.2cm}

 By definition, the zero set  $V_f:= f^{-1}(0)$ is called \textit{the mixed hypersurface}. It follows that $0 \in V_f$, where $0$ is the origin of $\C^2$. For a fixed variable $x \in \{u,v,\bar{u},\bar{v}\}$ we say that $f$ is a \textit{$x$-semiholomorphic polynomial} if $f$ does not depend on the variable $\bar{x}$, where naturally $\overline{(\bar{u})}=u$ and $\overline{(\bar{v})}=v$. With this definition we have that $f$ is holomorphic if and only if $f$ is both $u$-semiholomorphic and $v$-semiholomorphic. Analogously, $f$ is anti-holomorphic if and only if $f$ is both $\bar{u}$-semiholomorphic and $\bar{v}$-semiholomorphic. 

\vspace{0.2cm}

Following the convention in \cite{bode_lemniscate} we often refer to a $u$-semiholomorphic polynomial simply as semiholomorphic. We also denote the partial derivative of $f$ with respect to the variable $u$ by $f_u$. Analogously, we write $f_{\bar{u}}$, $f_v$ and $f_{\bar{v}}$ for the partial derivatives with respect to the variables $\bar{u}, v$ and $\bar{v}$, respectively. 
\begin{definition}\label{conjuntosingular}
Given a mixed polynomial $f:(\C^2,0)\to (\C,0)$, the singular set of $f$, denoted by $\Sigma_f,$ is defined as the solution of 
\begin{equation}\label{criticalpoint}
    \Sigma_f:=
\begin{cases} 
&s_{1,f}:= f_{u} \bar{f}_{\bar{v}}- \bar{f}_{\bar{u}}f_{v}=0, \\
 &s_{2,f}:=|f_{u}|^2-|f_{\bar{u}}|^2=0 \\
 &s_{3,f}:=|f_{v}|^2-|f_{\bar{v}}|^2=0
   \end{cases}
\end{equation}
considered as a germ of a set at the origin. 
\end{definition}
\noindent Definition~\ref{conjuntosingular} is equivalent to the Oka's definition in \cite[Proposition 8]{Oka2010} and Rudolph's definition in \cite[Machinery 4.1]{Rudolph1987}. If there is a neighbourhood $U$ of the origin $0\in \C^2$ with $U\cap\Sigma_f=\{0\}$, we say that the origin is an isolated mixed singularity of $f$. If there is a neighbourhood $U$ of the origin $0$ with $U\cap\Sigma_f\cap V_f=\{0\}$, we say that the origin is a weakly isolated mixed singularity of $f$. 

\vspace{0.2cm}

\noindent Note that if the origin is a (weakly) isolated mixed singularity of the germ $f:(\C^2,0)\to (\C,0)$, it is also a (weakly) isolated singularity of the respective real mapping germ $f:(\R^4,0)\to (\R^{2},0)$, as defined in the introduction.

\vspace{0.2cm}

The support of $f$ is defined as $supp(f):=\{\nu+\mu: \ c_{\nu,\mu}\neq 0\}\subset\mathbb{N}^2$. Given $\nu+\mu \in supp(f)$ we denote by $M_{\nu,\mu}:=M_{\nu,\mu}(z,\bar{z}) =c_{\nu,\mu} z^{\nu}\bar{z}^{\mu}$ the mixed monomial of the corresponding degrees in $z$ and $\bar{z}$. Let $P =(p_1, p_2)\in (\N^+)^2$ be a non-zero vector, that we call a \emph{weight vector}. The \emph{radial degree} of the monomial $M_{\nu,\mu}$ with respect to weight vector $P$ is defined by $rdeg_P (M_{\nu,\mu}):= \sum^2_{j=1} p_j (\nu_j + \mu_j).$
We call $f(z, \bar{z})$ a \emph{radially weighted homogeneous polynomial of radial type} $(P; d_r),$ if 
$d_{r}=rdeg_P (M_{\nu,\mu})$ for every $\nu+\mu \in supp(f).$


\vspace{0.2cm}

The \textit{Newton polygon} $\Gamma_{+}(f)$ of a mixed polynomial $f$ is the convex hull of
$$\bigcup_{\nu+\mu \in supp(f)}(\nu+\mu)+(\R^{+})^{2}.$$
See Figure~\ref{newtonboundaryholomorphicq} for an illustration of an example.
The \textit{Newton boundary} $\Gamma(f)$ of a Newton polygon is defined to be the set of lattice points $w\in\mathbb{Z}^2$ that lie on some compact face of $\Gamma_{+}(f).$ These faces are either 0-dimensional (``vertices'') or 1-dimensional (``edges''). The \textit{Newton principal part} $f_{\Gamma}$ of $f$ is defined by
\begin{equation}\label{eq:ffboun}
    f_{\Gamma}(z,\bar{z})=\sum_{\nu+\mu \in \Gamma(f)} c_{\nu,\mu} z^\nu \bar{z}^\mu.
\end{equation} 
When $f\equiv f_{\Gamma}$, we say that $f$ is a \textit{boundary polynomial}.


\vspace{0.2cm}
By definition, we say that $f(z,\bar{z})$ is $u-$\emph{convenient} if $\Gamma(f)$ intersects the horizontal axis of $(\R^{+})^{2}.$ Likewise, we say that $f(z,\bar{z})$ is $v$-convenient if $\Gamma(f)$ intersects the vertical axis of $(\R^{+})^{2}.$ A mixed function is called \emph{convenient} if it is both $u$-convenient and $v$-convenient.
\vspace{0.2cm}

For each compact face $\Delta$ (0- or 1-dimensional) we define the \emph{face function associated to the face $\Delta$} by $$f_{\Delta}(z,\bar{z})\defeq \sum_{\nu+\mu \in \Delta} c_{\nu,\mu}z^{\nu}\bar{z}^{\mu}.$$ 

For a given compact 1-dimensional face, or 1-face for short, $\Delta$ of $\Gamma_+(f)$ there exists a unique weight vector $P =(p_1, p_2)$ orthogonal to $\Delta$ with $\gcd(p_{1},p_{2})=1,$  
such that the linear function $\ell_P$ on $\Gamma (f)$ defined by $\ell_P(w):= p_1 w_1+p_2 w_2 ,$ for $w:=(w_1,w_2) \in \Gamma(f)$ takes its minimal value on $\Delta$. In this case we write $\Delta := \Delta (P, f)$ or, $\Delta(P)$ for short, and denote the minimal value by $d(P; f)$. We then also refer to $f_\Delta$ as the \emph{face function associated to the weight $P$} and denote it by $f_P$. Note that the face function $f_{\Delta(P)} (z,\bar{z})$ is radially weighted homogeneous of type $(P,d_{r}),$ with $d_{r} =d(P;f)$.

\begin{obs}More generally, for any non-zero weight vector $Q=(q_{1},q_{2})\in (\N^{+})^{2}$ one may consider the linear function $\ell_Q$ on a given Newton boundary $\Gamma(f)$ via $\ell_{Q}(x)=q_{1}a+q_{2}b,$ for $x=(a,b)\in\Gamma(f).$ Now, for a given mixed polynomial $f$ with Newton boundary $\Gamma(f)$ let $d(Q;f)$ be the minimal value that $\ell_{Q}$ takes on $\Gamma(f)$. The {\bf Q-relative face function associated to the Newton boundary $\Gamma(f)$} is defined and denoted by

\begin{equation*}
  f_{Q}(z,\bar{z}):=\sum_{\substack{\nu+\mu \in \Gamma(f),\\ \ell_{Q}(\nu+\mu)=d(Q;f)}}c_{\nu,\mu}z^{\nu}\bar{z}^{\mu}.
\end{equation*}
If $Q$ is the weight vector associated to some 1-face $\Delta(Q)$, then we have $f_Q=f_{\Delta(Q)}$.
\end{obs}
We say that a lattice point $w\in \mathbb{Z}^2$ lies \textit{above the Newton boundary} of $f$ if $\ell_Q(w)>d(Q;f)$ for all positive weight vectors $Q$.

\vspace{0.3cm}

On the set of non-zero weight vectors we consider the following total order: for $P=(p_1,p_2)\in(\mathbb{N}^+)^2$ and $Q=(q_1,q_2)\in(\mathbb{N}^+)^2$ with $\gcd(p_1,p_2)=\gcd(q_1,q_2)=1,$ we write $P\succ Q$ if and only if $\tfrac{p_1}{p_2}>\tfrac{q_1}{q_2}$.

On the Newton boundary $\Gamma(f)$ consider the sequence of positive weight vectors $\mathcal{P=}\{P_1, P_2,\ldots,P_N\}$ associated to the $N$ compact 1-faces of $\Gamma_+(f),$ indexed according to the order above so that $P_i\succ P_j$ if and only if $i<j$. This order is equivalent to ordering the compact 1-faces of $\Gamma_+(f)$ by their slope, with $\Delta(P_1)$ being the steepest edge of all compact 1-faces of $\Gamma_+(f)$, $\Delta(P_2)$ having the next larger negative slope, and so on. In the next example we will illustrate the main concepts given above.


\begin{ex}\label{ex1}

We consider the boundary polynomial $$f(u,v,\bar{v})= u^8+v^3u^2+\bar{v}^5u-2(v^7+\bar{v}^7).$$ 

 \begin{figure}[H]
 \centering
 \begin{tikzpicture}[scale=0.8]
\begin{axis}[axis lines=middle,axis equal,yticklabels={0,,2,4,6,8,10}, xticklabels={0,,2,4,6,8,10},domain=-10:10,     xmin=0, xmax=10,
                    ymin=0, ymax=8,
                    samples=1000,
                    axis y line=center,
                    axis x line=center]
\addplot coordinates{(11,0)(8,0) (2,3) (1,5) (0,7) (0,10)};
\draw[line width=2pt,red,-stealth,thick,dashed](50,15)--(60,35)node[anchor=west]{$\boldsymbol{P_2}$};
\draw[line width=2pt,red,-stealth,thick,dashed](10,50)--(30,60)node[anchor=west]{$\boldsymbol{P_1}$};
 \filldraw[black] (50,110) node[anchor=north ] {$\Gamma_+(f)$};
  \filldraw[black] (65,65) node[anchor=north ] {$\Gamma_+(f)$};
\filldraw[blue] (20,30) circle (2pt) node[anchor=south west] {$\boldsymbol{\Delta_2}$};
\filldraw[blue] (0,70) circle (2pt) node[anchor=south west] {$\boldsymbol{\Delta_{1}}$};
\filldraw[blue] (80,0) circle (2pt) node[anchor=south west] {$\boldsymbol{\Delta_{3}}$};
\fill[yellow!90,nearly transparent] (0,70) -- (20,30) -- (80,0) -- (260,0) -- (260,300) -- (0,300) --cycle;
\end{axis}
\end{tikzpicture}
\caption{The Newton polygon $\Gamma_+(f)$ of a boundary polynomial $f$.
\label{newtonboundaryholomorphicq}}
\end{figure}
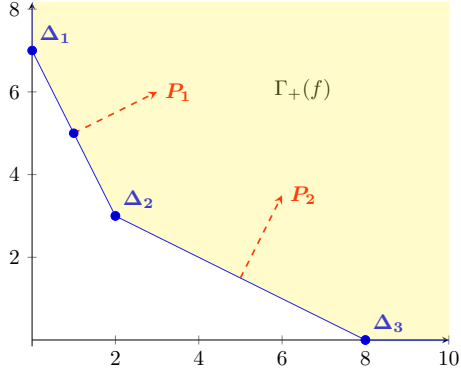

Its Newton boundary $\Gamma (f)$, depicted in Figure \ref{newtonboundaryholomorphicq}, has two compact 1-faces and three vertices $\Delta_1,\Delta_2$ and $\Delta_3$. Corresponding to the 1-faces we have the sequence of weight vectors $\mathcal{P}=\{P_1=(2,1),P_2=(1,3)\}$. The polynomial is convenient, since $\Delta_1$ lies on the vertical axis and $\Delta_3$ lies on the horizontal axis. It is also $u$-semiholomorphic (a property that cannot be verified by only looking at its Newton polygon). We have
\begin{align}
f_{P_1}(z,\bar{z})&=v^3u^2+\bar{v}^5u-2(v^7+\bar{v}^7),\nonumber\\
f_{P_2}(z,\bar{z})&=u^8+v^3u^2.
\end{align}
For any other non-zero weight vector $Q$ in this case the Q-relative face functions associated to the Newton boundary $\Gamma(f)$ are given by:
\begin{align}
f_{Q}&=f_{\Delta_1}=-2(v^7+\bar{v}^7)&\qquad\text{ if }Q\succ P_1,\nonumber\\
f_{Q}&=f_{\Delta_2}=v^3u^2&\qquad\text{ if }P_1\succ Q\succ P_2,\nonumber\\
f_{Q}&=f_{\Delta_3}=u^8&\qquad\text{ if }P_2\succ Q.
\end{align}
 
\end{ex}


\vspace{0.3cm}

Recall that every compact 1-face $\Delta(P_i)$ corresponds to a radially weighted homogeneous polynomial $f_{P_i}$. Writing $v=r\rme^{\rmi t}$, $\bar{v}=r\rme^{-\rmi t},$ we may associate to $\Delta(P_i)$ a function that does not depend on $r$ anymore, as follows:  for $P_i=(p_{i,1},p_{i,2}) \in \mathcal{P}$ let $g_i:\mathbb{C}\times S^1\to\mathbb{C}$ be such that
\begin{equation}\label{resc}
f_{P_i}(u,\bar{u},v,\bar{v})=f_{P_i}(u,\bar{u},r\rme^{\rmi t},r\rme^{-\rmi t})=r^{k_is_i+n_i} g_i\left(\frac{u}{r^{k_i}},\frac{\bar{u}}{r^{k_i}},\rme^{\rmi t}\right),
\end{equation}
where $k_i:=\tfrac{p_{i,1}}{p_{i,2}}$, $n_i$ is the smallest exponent of $r=|v|=|\bar{v}|$ in $f_{P_i}$ and $s_i$ is the largest exponent of $R:=|u|=|\bar{u}|$ in $f_{P_i}$. Of course, this is equivalent to 
\begin{equation}
g_i(u,\bar{u},\rme^{\rmi t})=\frac{f_{P_i}(r^{k_i}u,r^{k_i}\bar{u},r\rme^{\rmi t},r\rme^{-\rmi t})}{r^{k_is_i+n_i}}
\end{equation}
for any $r>0$. Note that $g_i$ is defined in such a way that it does not depend on the variable $r$. 

\vspace{0.2cm}

As in \cite{Oka2010} we call $f_{P_i}$ \emph{true} if it has a non-empty vanishing set in $(\C^*)^2$, where $\C^*=\C\backslash\{0\}$. This is equivalent to the corresponding polynomial $g_i(u,\bar{u},\rme^{\rmi t})$ having a non-empty vanishing set in $\mathbb{C}^*\times S^1.$

\vspace{0.2cm}

A mixed polynomial $f:\mathbb{C}^2\to\mathbb{C}$ is called \emph{true} if for all $P_i\in \mathcal{P}$ the respective $f_{P_i}$ is true.

\begin{definition}[Oka \cite{Oka2010}]
A face function $f_{\Delta}$ of a mixed polynomial $f$ for some compact face $\Delta$ (0- or 1-dimensional) of $\Gamma_+(f)$ is called {\bf Newton non-degenerate} if $V_{f_{\Delta}}\cap\Sigma_{f_{\Delta}}\cap(\mathbb{C}^*)^2=\emptyset$. It is called {\bf strongly Newton non-degenerate} if $\Sigma_{f_{\Delta}}\cap(\mathbb{C}^*)^2=\emptyset$. We say that $f$ is (strongly) Newton non-degenerate if $f_{\Delta}$ is (strongly) Newton non-degenerate for all compact faces $\Delta$ of $\Gamma_+(f)$.
\end{definition}

Occasionally, we will refer to (strongly) Newton non-degenerate functions simply as (strongly) non-degenerate.

\vspace{0.2cm}

Oka proved in \cite{Oka2010} that convenient, Newton non-degenerate mixed functions have weakly isolated singularities at the origin and convenient, strongly non-degenerate mixed functions have isolated singularities at the origin. In either case the intersection of $V_f$ and a 3-sphere $S^3_{\rho}$ of sufficiently small radius $\rho>0$ produces a well defined link $L_f,$ up to isotopy, called {\it the link of the singularity $f$}. If $f$ has a weakly isolated singularity, we say that $f$ is a \textit{weak realization} of $L_f$. If $f$ has an isolated singularity, we say that $f$ is a \textit{strong realization} of $L_f$.

\vspace{0.2cm}

Oka also found a relation between the number of components of the link of the singularity and the numbers of components of links associated to each face function \cite{Oka2010}. In the later sections, we will see that this relation goes much deeper than just the number of components. In fact, the isotopy class of the link of the singularity can be completely determined by these links associated with the face functions.

\subsection{Braids}
In this subsection we recall the fundamentals on braids. A standard reference is \cite{kassel}. 
Let ${u_1,u_2,\ldots,u_s}$ be a set of $s$ distinct complex numbers. A \textit{geometric braid on $s$ strands} is a collection of $s$ disjoint curves in $\mathbb{C}\times[0,2\pi]$ (where the interval is  interpreted as the vertical direction) without any horizontal tangents and with endpoints $(u_j,0)$, $j\in\{1,2,\ldots,s\}$ and $(u_j,2\pi)$, $j\in\{1,2,\ldots,s\}$. See Figure~\ref{fig:braidex}-a) for an example.

It follows from this definition that all of these curves can be parametrized by their height coordinate $t\in[0,2\pi]$. That is, a geometric braid is given explicitly by the union
\begin{equation}\label{eq:braidpara}
    \bigcup_{t\in[0,2\pi]}\bigcup_{j=1}^s\{(u_j(t),t)\}
\end{equation}
for some smooth functions $u_j:[0,2\pi]\to\mathbb{C}$, $j=1,2,\ldots,s$, with $u_j(0)=u_j$ for all $j$, with $u_j(t)\neq u_{j'}(t)$ for all $t\in[0,2\pi]$ and all $j\neq j'$, and with $u_j(2\pi)=u_{\pi(j)}(0)$ for some permutation $\pi$ of $\{1,2,\ldots,s\}$ and all $j\in\{1,2,\ldots,s\}$.

\begin{figure}[h!]
\centering
\labellist
\Large
\pinlabel a) at 10 385
\pinlabel b) at 205 385
\pinlabel c) at 515 385
\endlabellist
\includegraphics[height=6cm]{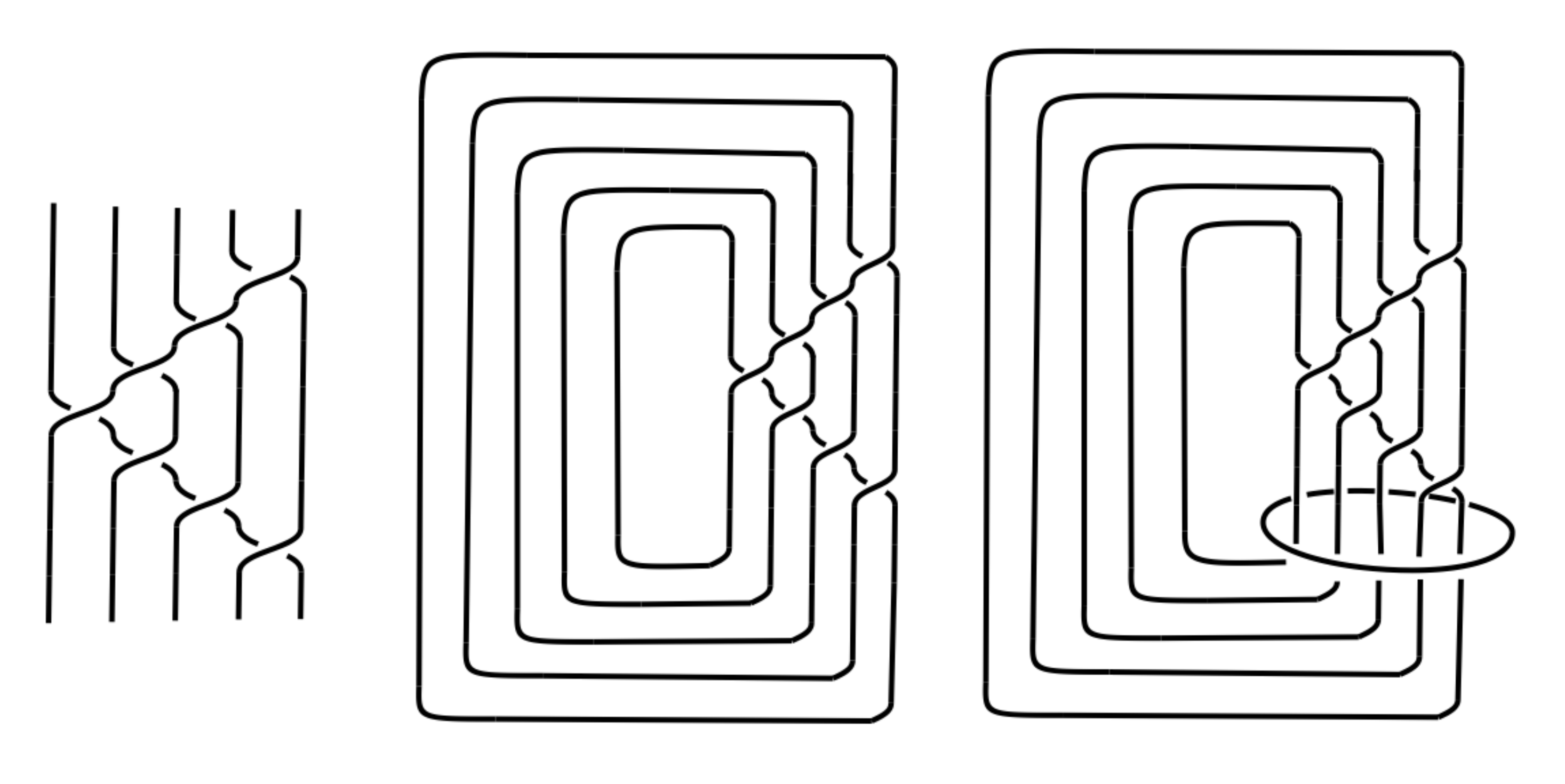}
\caption{a) The diagram of a braid. b) The diagram of the closure of a braid. c) The union of a closure of a braid and its braid axis.\label{fig:braidex}}
\end{figure}

Two geometric braids $B$ and $B'$ on $s$ strands are called \textit{braid isotopic} (or simply \textit{isotopic}) if there is a smooth family of geometric braids $B_{\chi}$, $\chi\in[0,1]$, with $B_0=B$, $B_1=B'$ and such that the endpoints of the strands $(u_j,0)$ and $(u_j,2\pi)$, $j=1,2,\ldots,s$, are fixed throughout the isotopy. In other words, the two sets of parametric curves are related by a smooth isotopy that maintains the geometric braid property throughout the deformation without changing the endpoints. Braid isotopies define an equivalence relation on the set of geometric braids for a fixed choice of ${u_1,u_2,\ldots,u_s}$. We call an equivalence class a \textit{braid}. In many texts both the equivalence class and any representative of that class are called a braid. However, in this article some constructions depend on the actual geometry of the parametric curves as opposed to the topological properties of its equivalence class. In order to emphasize this, we refer to an isotopy class as a braid and call a representative of that isotopy class a geometric braid.

A geometric braid is called \textit{affine} (see for example \cite{OR, ZW}) or \textit{a geometric braid in $(\C\backslash\{0\})\times[0,2\pi]$} if $u_j(t)\neq 0$ for all $j=1,2,\ldots,s$ and all $t\in[0,2\pi]$. Two affine geometric braids are braid isotopic as affine braids if they are related by a braid isotopy in $(\C\backslash\{0\})\times[0,2\pi]$, that is, throughout the isotopy all intermediate geometric braids are affine. The equivalence class of an affine geometric braid under this equivalence relation is called its affine braid istopy class.

Identifying the $t=0$-plane and the $t=2\pi$-plane turns $\mathbb{C}\times[0,2\pi]$ into the (open) solid torus $\mathbb{C}\times S^1$. Since the endpoints of a geometric braid $B$ on the $t=0$-plane match its endpoints on the $t=2\pi$-plane, this identification turns a geometric braid into a collection of disjoint simple closed curves, a link.

With the term \textit{closed braid} we usually refer to the link in $\C\times S^1$ and we denote it by $B$, the same as the (geometric) braid in $\C\times[0,2\pi]$. A closed braid $B$ is thus a link in $\C\times S^1$ such that the projection map $\C\times S^1\to S^1$ onto the second factor restricts to a covering map $B\to S^1$. We say that two closed braids $B_1$, $B_2$ are isotopic as closed braids if there is a smooth isotopy of closed braids $K_{\chi}$, $\chi\in[0,1]$, in $\C\times S^1$ with $K_0=B_1$ and $K_1=B_2$. This can be interpreted as a braid isotopy in $\C\times [0,2\pi]$ where the chosen complex numbers $u_1, u_2,\ldots,u_s$ are allowed to vary smoothly.

Since we are specifically interested in links in the 3-sphere, we use the standard untwisted embedding of the solid torus in $S^3$ as the neighbourhood of an unknot to obtain a link in $S^3$, which is called \textit{the closure of the geometric braid} $B$, see Figure~\ref{fig:braidex}b). Note the subtlety in our terminology. A closed braid is a link in $\C\times S^1$, but the closure of a braid is a link in $S^3$. It is well-known that every link in $S^3$ arises as the closure of some braid. Furthermore, the link type of the closure only depends on the braid (the braid isotopy class), not the geometric braid (its representative).

\vspace{0.2cm}

The complement of this embedded (open) solid torus in $S^3$ is a (closed) solid torus. We call the core (the center line, the 0-section, ...) $\gamma$ of this complementary solid torus a \textit{braid axis} of the closure of $B$. Note that $\gamma$ is an unknot. The union of a closed braid and its braid axis is depicted in Figure~\ref{fig:braidex}c).

\vspace{0.2cm}

This can all be made very explicit. If a geometric braid satisfies $|u_j(t)|<1$ for all $t\in[0,2\pi]$ and $j=1,2,\ldots,s$, we can simply project the strands onto $S^3$, so that the corresponding braid closure is given by 
\begin{equation}
\bigcup_{t\in[0,2\pi]}\bigcup_{j=1}^s(u_j(t),\sqrt{1-|u_j(t)|^2}\rme^{\rmi t})\subset S^3\subset\mathbb{C^2}.
\end{equation} 
The curve $\cup_{\chi\in[0,2\pi]}(\rme^{\rmi \chi},0)$ would be an example of a braid axis.

In this article we employ the connection between braids and complex polynomials, or more precisely, between braids and loops in certain spaces of polynomials. 

\vspace{0.2cm}

For a given geometric braid $B$ on $s$ strands, parametrized as in Eq.~(\ref{eq:braidpara}) we can define a 1-parameter family of polynomials $g_t:\mathbb{C}\to\mathbb{C}$, $t\in[0,2\pi]$, via
\begin{equation}
    g_t(u)=\prod_{j=1}^s(u-u_j(t)).
\end{equation}
Note that for every value of $t$, this is a monic complex polynomial of degree $s$ with distinct roots $u_j(t)$, $j=1,2,\ldots,s$, which are precisely the complex coordinates of the $s$ strands of $B$ at height $t$. Therefore, as $t$ varies from 0 to $2\pi$, the roots of $g_t$ trace out the geometric braid $B$.

By definition we have that $g_0=g_{2\pi}$ and hence $g_t$ can be thought of as a loop in the space $X_s$ of monic complex polynomials of degree $s$ with distinct roots. In fact, the fundamental theorem of algebra allows us to go in the other direction: The roots of any loop $g_t$ in $X_s$ trace out a geometric braid on $s$ strands and every geometric braid on $s$ strands is realized as the roots of a loop in $X_s$. Furthermore, isotopy classes of geometric braids form a group, the braid group on $s$ strands, which from this perspective can be interpreted as the fundamental group of $X_s$. Geometrically the group operation is given by concatenation (and rescaling) of two sets of curves with the same endpoints $u_1,u_2,\ldots,u_s$. Isotopies of closed braids in $\C\times S^1$ can be understood as homotopies of loops in $X_s$ that do not fix the basepoint $\prod_{j=1}^s(u-u_j)$.

Given a loop $g_t$ in $X_s$ it is often advantageous to interpret it as a map $g:\mathbb{C}\times S^1\to\mathbb{C}$, 
\begin{equation}
    g(u,\rme^{\rmi t})\defeq g_t(u).
\end{equation}

In previous work by the second author this relation has been employed to construct real polynomial maps with isolated singularities and links of singularities that are closures of certain families of braids \cite{Bode2019,Bode2020,bodesat,bodebraid}. If the strands $u_j(t)$ are given in terms of trigonometric polynomials (a condition that can always be satisfied after a $C^1$-small deformation), the function $g$ above is a polynomial in the complex variable $u$, but also in $\rme^{\rmi t}$ and $\rme^{-\rmi t}$.

We can traverse a given loop $g_t$ in $X_s$ whose roots form the geometric braid $B$ several times and obtain loop $g_{m t}$, $m\in\mathbb{Z}$, in $X_s$ whose roots form $B^m$, the $m$th power of $B$, that is, $m$ copies of $B$ concatenated. Traversing a loop in the opposite direction results in the inverse braid $B^{-1}$. We have $B^{-m}\defeq (B^{-1})^m=(B^m)^{-1}$.

If a geometric braid is of the form $B^2$ for some geometric braid $B$, then all exponents of $\rme^{\rmi t}$ and $\rme^{-\rmi t}$ in the corresponding polynomial $g$ are even. We then define the semiholomorphic, radially weighted homogeneous polynomial $f:\mathbb{C}^2\to\mathbb{C}$ via
\begin{equation}
f(u,r\rme^{\rmi t})=r^{2ks}g\left(\frac{u}{r^{2k}},\rme^{\rmi t}\right),
\end{equation}
with $k$ chosen sufficiently large.

\begin{obs} Two important consequences to point out: 
\begin{itemize}
    \item[(1)]  Since all exponents of $\rme^{\rmi t}$ and $\rme^{-\rmi t}$ in $g$ are even, $f$ can be written as a mixed polynomial with $v=r\rme^{\rmi t}$ and $\bar{v}=r\rme^{-\rmi t}$. Since all roots of $g_t$ are simple for all $t$, the mixed polynomial $f$ has a weakly isolated singularity at the origin.

\item[(2)] The advantage of working with semiholomorphic polynomial becomes apparent now. Since $f$ is holomorphic with respect to $u$, a critical point $(u_*,v_*)\in\C\times\C^*$ would have to satisfy $f_u(u_*,v_*)=0$, which implies $g_u(u_*,\rme^{\rmi t_*})=0$, where $v_*=r_*\rme^{\rmi t_*}$. In other words, the $u$-coordinate of a critical point $(u_*,v_*)$ of $f$ must be a critical point of the complex polynomial $g_{t_*}:\C\to\C$. Since $g_t$ is a polynomial of degree $s$, it has exactly $s-1$ critical points, when counted with multiplicity. We denote them by $c_j(t)$, $j=1,2,\ldots,s-1$. It turns out that the determinant of the $2\times 2$ real Jacobian matrix of $f$ associated to derivatives with respect to $r$ and $t$ is a non-zero multiple of 
\begin{equation}\label{eq:argcrit}
\frac{\partial \arg(g_t(c_j(t)))}{\partial t}.
\end{equation} 
Thus, $f$ has no critical points in $\C\times\C^*$ if the derivative above has no zeros. Geometric and topological interpretations of this fact have been put forward in \cite{bodesat, bodebraid}.
\end{itemize}
\end{obs}

\begin{definition}(\cite{bodesat})A geometric braid $B$ is called {\bf P-fibered} if the corresponding function $g$ induces a fibration via $\arg(g):(\mathbb{C}\times S^1)\backslash B\to S^1$.
A braid is called P-fibered if it can be represented by a P-fibered geometric braid.
\end{definition}
Here $\arg(z)$ maps a non-zero complex number $z=R\rme^{\rmi \chi}$ to its complex argument $\chi$. The term fibration here refers to the absence of any critical points of $\arg(g)$, i.e., there is no point in $(\mathbb{C}\times S^1)\backslash B$ at which all three independent directional derivatives of $\arg(g)$ vanish. Since $g$ is holomorphic with respect to $u$, this condition is equivalent to the non-vanishing of Eq.~(\refeq{eq:argcrit}). Therefore, $f$ as above has an isolated singularity if and only if $B^2$ is a P-fibered geometric braid \cite{Bode2019}. It follows from the definition that a geometric braid $B$ is P-fibered if and only if $B^n$ is a P-fibered geometric braid for all non-zero integers $n$.

Hence, the closure of $B^2$ is a real algebraic link if $B$ is P-fibered. Different constructions in \cite{Bode2019,Bode2020,bodesat,bodebraid} of real algebraic links often amount to proving that a certain family of braids is P-fibered, the best known example among these being the family of homogeneous braids.

\section{INNER NON-DEGENERATE MIXED POLYNOMIALS}\label{section3}
In this section we introduce the notion of an inner non-degenerate mixed function. We study its relation to Oka's non-degeneracy and show that inner non-degenerate mixed polynomials have a weakly isolated singularity.

Let $f:\C^2\to\C$ be a mixed polynomial with Newton polygon $\Gamma_+(f)$, whose boundary has at least one compact 1-face. A 0-dimensional face of the boundary of $\Gamma_+(f)$, that bounds a compact 1-face, is called an \textit{extreme vertex} if it is the boundary of a unique 1-face. Otherwise, we call it a \textit{non-extreme vertex}. Let $\{P_1,P_2,\ldots,P_N\}$ be the sequence of weight vectors corresponding to the compact 1-faces of $\Gamma_+(f)$, ordered as in the previous section. With these notions we consider the following definition.
\begin{definition}\label{Newtoncond}
We say that $f$ has an inner Newton non-degenerate boundary if both of the following conditions hold:
\begin{itemize}
    \item[(i)] the face functions $f_{P_1}$ and $f_{P_N}$ have no critical points in $V_{f_{P_1}} \cap (\C^2\setminus \{v=0\})$ and $V_{f_{P_N}}\cap (\C^2\setminus \{u=0\})$, respectively.
\item[(ii)] on each 1-face and non-extreme vertex $\Delta$, the face function $f_\Delta$ has no critical points in $V_{f_\Delta} \cap (\C^*)^{2}$.
\end{itemize}
\end{definition}
As in the other notions of non-degeneracy, we sometimes refer to $f$ itself as being inner non-degenerate if it has an inner Newton non-degenerate boundary.
\begin{obs}
The following remarks follow directly from this definition:
\item[1.] For the ordered sequence $\{P_1,P_2,\ldots,P_N\}$ of weight vectors corresponding to the compact 1-faces of $\Gamma_+(f)$ and $k_i=\tfrac{p_{i,1}}{p_{i,2}}$ i=1,N, as in the expression (\ref{resc}), consider any positive weight vector $P=(p_1,p_2)$ with $k_{1}\geq \frac{p_{1}}{p_{2}}\geq k_{N}$. Then, Condition (ii) implies that the $P$-relative face function $f_P$ has no critical point in $V_{f_P}\cap (\C^*)^{2}$.  
\item[2.] The Condition (ii) is the Newton non-degeneracy condition as in \cite[ Definition 3]{Oka2011}, except that on the extreme vertices the condition may fail.
\item[3.] If $f$ is holomorphic, Definition~\ref{Newtoncond} is equivalent to the definition of weakly inner non-degenerate polynomials \cite{stevens}, which is equivalent to partial non-degeneracy in \cite{mondal} and \cite{wall} for germ of functions $(\C^{2},0)\to (\C,0)$, and called inner non-degeneracy in \cite{boubakri}.
\end{obs}

\begin{prop}\label{lem:inner}
Let $f$ be a Newton non-degenerate, convenient mixed polynomial. Then $f$ is inner non-degenerate.
\end{prop}
\begin{proof}
Suppose that $f$ is inner degenerate, that is, not inner non-degenerate. Since $f$ is Newton non-degenerate, Condition (ii) in Definition~\ref{Newtoncond} is satisfied and hence Condition (i) in Definition~\ref{Newtoncond} is not satisfied. Therefore, there exists a critical point of $f_{P_1}$ in $V_{f_{P_1}}\cap \mathbb{C}^2\backslash\{v=0\}$ or a critical point of $f_{P_N}$ in $V_{f_{P_N}}\cap \mathbb{C}^2\backslash\{u=0\}$. We assume the former holds true. The argument in the latter case is analogous.

Let $(u_*,v_*)\in V_{f_{P_1}}\cap \mathbb{C}^2\backslash\{v=0\}$ be a critical point of $f_{P_1}$. Since $f$ is non-degenerate, $f_{P_1}$ has no critical point in $V_{f_{P_1}}\cap \mathbb{C}^{*2}$ and hence $u_*=0$. Since $f$ is convenient, its Newton boundary intersects the vertical axis and we call the corresponding vertex $\Delta$. Since $v_*\neq 0$, we write $f_{\Delta}$ in a neighbourhood of $v_*$ as $r^k\Phi(\rme^{\rmi t})$ for some $\Phi:S^1\to\mathbb{C}$, as usual writing the complex variable $v$ as $r\rme^{\rmi t}$. Note that $f_{\Delta}(u,v)=f_{P_1}(0,v)$ for all $u\in\mathbb{C}$.

With $v_*\neq 0$ we have that $f_{\Delta}(u,v_*)=f_{P_1}(0,v_*)=0$ implies that $\Phi(\rme^{\rmi t_*})=0$. We calculate the real Jacobian matrix of $f_{\Delta}$ at $(u,v_*)=(u,r_*\rme^{\rmi t_*})$:
\begin{align}
Df_{\Delta}(u,v_*)&=
\begin{pmatrix}
\frac{\partial \text{Re}(f_{\Delta})}{\partial \text{Re}(u)}(u,v_*) &\frac{\partial \text{Re}(f_{\Delta})}{\partial \text{Im}(u)}(u,v_*) &\frac{\partial \text{Re}(f_{\Delta})}{\partial r}(u,v_*) &\frac{\partial \text{Re}(f_{\Delta})}{\partial t}(u,v_*)\\
\frac{\partial \text{Im}(f_{\Delta})}{\partial \text{Re}(u)}(u,v_*) &\frac{\partial \text{Im}(f_{\Delta})}{\partial \text{Im}(u)}(u,v_*) &\frac{\partial \text{Im}(f_{\Delta})}{\partial r}(u,v_*) &\frac{\partial \text{Im}(f_{\Delta})}{\partial t}(u,v_*)
\end{pmatrix}\nonumber\\
&=\begin{pmatrix}
0 & 0 & kr_*^{k-1}\text{Re}(\Phi)(\rme^{\rmi t_*}) & r_*^k\frac{\partial \text{Re}(\Phi)}{\partial t}(\rme^{\rmi t_*})\\
0 & 0 & kr_*^{k-1}\text{Im}(\Phi)(\rme^{\rmi t_*}) & r_*^k\frac{\partial \text{Im}(\Phi)}{\partial t}(\rme^{\rmi t_*})
\end{pmatrix}\nonumber\\
&=\begin{pmatrix}
0 & 0 & 0 & r_*^k\frac{\partial \text{Re}(\Phi)}{\partial t}(\rme^{\rmi t_*})\\
0 & 0 & 0 & r_*^k\frac{\partial \text{Im}(\Phi)}{\partial t}(\rme^{\rmi t_*})
\end{pmatrix},
\end{align}
which follows from the fact that $f_{\Delta}$ does not depend on $\text{Re}(u)$ or $\text{Im}(u)$ and $\Phi(\rme^{\rmi t_*})=0$. Since $Df_{\Delta}$ does not have full rank, $(u,v_*)$ is a critical point of $f_{\Delta}$ for any complex number $u$. In particular choosing $u\neq 0$ results in a critical point of $f_{\Delta}$ in $V_{f_{\Delta}}\cap(\mathbb{C}^*)^{2}$, which contradicts the Newton non-degeneracy of $f$. 
\end{proof}

The convenience condition in Proposition~\ref{lem:inner} cannot be dropped, since $u^4-u^2v^3$ is an example of an inconvenient Newton non-degenerate polynomial that is inner degenerate. There are inconvenient inner non-degenerate polynomials, but their extreme vertices can only have distance 1 from the coordinate axis.

The following example illustrates that there are inner non-degenerate mixed polynomials that are Newton degenerate.
\begin{ex}\label{ex2}

We consider again the boundary polynomial $$f(u,v,\bar{v})= u^8+v^3u^2+\bar{v}^5u-2(v^7+\bar{v}^7)$$ from Example~\ref{ex1}. The vertices $\Delta_1$ and $\Delta_3$ are the extreme vertices of $f$ and $\Delta_2=\Delta(P_1)\cap \Delta(P_2)$ is the unique non-extreme vertex. 
\vspace{0.3cm}

Since $f$ is semiholomorphic, the intersection of the singular set and the variety of $f$ is given by the solution to $f_u=s_{3,f}=f=0$. 
\vspace{0.3cm}

We represent the face functions $f_{P_1}$ and $f_{P_2}$ as in Eq.~(\ref{resc}) as follows. As usual we write $v=re^{it}$. Associated to the weight vector $P_1$ we have $k_1=\tfrac{p_{1,1}}{p_{1,2}}=2,\ s_1=2$ and $n_1=3$. Then Eq.~\eqref{resc} takes the form 
\begin{align}
f_{P_1}(z,\bar{z})&=v^3u^2+\bar{v}^5u-2(v^7+\bar{v}^7)\nonumber\\
&= r^7 \left( e^{3it}\left( \frac{u}{r^2}\right)^2+ e^{-5it}\left(\frac{u}{r^2} \right)-2(e^{7it}+e^{-7it})  \right),
\end{align} 
so that
\begin{equation}
g_1(u,\rme^{\rmi t})= e^{3it}u^2+ e^{-5it}u-2(e^{7it}+e^{-7it})  
\end{equation}

Associated to weight vector $P_2$ we have $k_2=\tfrac{p_{2,1}}{p_{2,2}}=\tfrac{1}{2},\ s_2=8$ and $n_2=0$. Then Eq.~\eqref{resc} takes the form 
\begin{align}
f_{P_2}(z,\bar{z})&=u^8+v^3u^2\nonumber\\
&=r^{4}\left(\left( \frac{u}{r^{1/2}}\right)^8+ e^{3it}\left(\frac{u}{r^{1/2}} \right)^2 \right),
\end{align}
so that 
\begin{equation}
g_2(u,\rme^{\rmi t})=u^8+ e^{3it}u^2.
\end{equation}
  
Let $c(t)=-\tfrac{1}{2}\rme^{-8\rmi t}$ denote the critical point of $g_1(\cdot,\rme^{\rmi t})$. Since $g_1(c(t),\rme^{\rmi t})\neq 0$ for all $t\in[0,2\pi]$, there is no critical point $(u_*,v_*)$ of $f_{P_1}$ with $v_*\neq0$ and $f_{P_1}(u_*,v_*)=0$. Since $f_{P_2}$ is holomorphic, we can easily check that all of its critical points $(u_*,v_*)$ satisfy $u_*=0$. Hence, $f$ satisfies Condition (i) in Definition~\ref{Newtoncond}.
  
\vspace{0.3cm}

To prove Condition (ii) we only need to show that $f_{\Delta_2}(u,v,\bar{v})=v^3u^2$ has no critical points in $(\C^*)^{2}$, which is trivial.
\vspace{0.3cm}

Therefore, $f$ has an inner non-degenerate boundary. However, $f_{\Delta_1}=-2(v^7+\bar{v}^7)$ is degenerate. 
\end{ex}

Recall that $d(P;f)$ denotes the minimal radial degree of any monomial of $f$ with respect to a weight vector $P$, or in other words the minimal value of $\ell_P$ on $\Gamma(f)$.
\begin{lemma}\label{facefunc}
Let $x\in\{u,\bar{u},v,\bar{v}\}$. Let $f:\C^2\to \C$ be a mixed polynomial which is not $\bar{x}$-semiholomorphic (where we take $\overline{(\bar{u})}=u$ and $\overline{(\bar{v})}=v$), then for every weight vector $P=(p_1,p_2)$ it follows that $d(P;f_x)\geq d(P;f)-p_i$, where $i=1$ if $x\in\{u,\bar{u}\}$ and $i=2$ if $x\in\{v,\bar{v}\}$. Moreover, the following are equivalent:
\begin{itemize}
\item[i)] $d(P;f_x)= d(P;f)-p_i$, where $i=1$ if $x\in\{u,\bar{u}\}$ and $i=2$ if $x\in\{v,\bar{v}\}$.
\item[ii)] $(f_x)_P=(f_P)_x$.
\item[iii)] $f_P$ is not $\bar{x}$-semiholomorphic.
\end{itemize}
\end{lemma} 
\begin{proof}
As in the statement of the lemma we set $i=1$ if $x\in\{u,\bar{u}\}$ and $i=2$ if $x\in\{v,\bar{v}\}$.
Taking the derivative with respect to $x$ reduces the degree of a monomial $M_{\nu,\mu}=c_{\nu,\mu}z^{\nu}\bar{z}^{\mu}$ with respect to the variable $x$ by one if $M_{\nu,\mu}$ depends on $x$, i.e., $\deg_x (M_{\nu,\mu})_x=\deg_x M_{\nu,\mu}-1$ if $\nu_1>0$. Therefore, the radial degree $rdeg_P(M_{\nu,\mu})$ of $M_{\nu,\mu}$ decreases by $p_i$ under differentiation by $x$ if $M_{\nu,\mu}$ depends on $x$.

We may write $f=f_P+f'$ for some mixed function $f'$ with $d(P;f')>d(P;f_P)=d(P;f)$. Clearly this implies $f_x=(f_P)_x+f'_x$ and $d(P;f_x)=\min\{d(P;(f_P)_x),d(P;f'_x)\}$. 

Suppose that $f_P$ is not $\bar{x}$-semihololmorphic. If $f'$ does not depend on $x$, we have $f_x=(f_P)_x$ and since $\ell_P$ is constant on $(f_P)_x$, we have $(f_x)_P=f_x=(f_P)_x$ and 
\begin{equation}
d(P;f_x)=d(P;(f_P)_x)=d(P;f_P)-p_i=d(P;f)-p_i.
\end{equation} 
If $f$ depends on $x$, then 
\begin{equation}
d(P;(f_P)_x)=d(P;f_P)-p_i<d(P;f')-p_i=d(P;f'_x)
\end{equation}
and hence $(f_x)_P=(f_P)_x$ and again 
\begin{equation}
d(P;f_x)=d(P;(f_P)_x)=d(P;f_P)-p_i=d(P;f)-p_i.
\end{equation} 
This proves that $iii)$ implies $ii)$ and $i)$.

Suppose that $f_P$ is $\bar{x}$-semiholomorphic. Then $f_x=f'_x$. Since $f$ itself is not $\bar{x}$-semiholomorphic, $f_x$ and $(f_x)_P$ are not constant zero, which implies that $(f_x)_P\neq(f_P)_x=0$ and 
\begin{equation}
d(P;f_x)=d(P;f'_x)=d(P;f')-p_i>d(P;f_P)-p_1=d(P;f)-p_i.
\end{equation} 
Hence, $i)$ and $ii)$ independently imply $iii)$. It also shows that if $d(P;f_x)\neq d(P;f)-p_i$, then $d(P;f_x)>d(P;f)-p_i$, which finishes the proof of the lemma.
\end{proof}

\begin{prop}\label{weak-isolated}
Let $f:\C^2\to \C$ be a mixed polynomial with an inner Newton non-degenerate boundary. Then $f$ has a weakly isolated singularity at the origin.
\end{prop}
\begin{proof}
We are going to prove the proposition by contradiction, using the curve selection lemma: If the singularity at the origin is not weakly isolated, then there exists a non-constant analytic curve in $\Sigma_f \cap V_f$, starting at the origin. There are several cases to consider. First, we suppose the curve is of the form $$z(\tau)=(a\tau^{p_1}+h.o.t., b\tau^{p_2}+h.o.t.)\in \Sigma_f \cap V_f \cap (\C^*)^{2}$$
with $a,b\neq 0$ and $k_1\geq\tfrac{p_1}{p_2}\geq k_N,$ for $k_{1},k_{N}$ as in (\ref{resc}). By definition we have $s_{1,f}(z(\tau))=s_{2,f}(z(\tau))=s_{3,f}(z(\tau))=f(z(\tau))=0.$ Let $P=(p_1,p_2)$. As in the proof of the previous lemma, we may express $f$ as the sum of $f_P$ and terms with higher radial degree with respect to $P$. We proceed analogously with $f_x$, $x\in\{u,\bar{u},v,\bar{v}\}$. Thus the four equations $s_{1,f}(z(\tau))=s_{2,f}(z(\tau))=s_{3,f}(z(\tau))=f(z(\tau))=0$ become  
\begin{align}
     ((f_{u})_P(a,b)\cdot (\bar{f}_{\bar{v}})_P(a,b)) \tau^{d(P;f_u)+d(P;f_{\bar{v}})}+h.o.t.&\nonumber =\\ =((f_{v})_P(a,b) \cdot (\bar{f}_{\bar{u}})_P(a,b)&)\tau^{d(P;f_v)+d(P;f_{\bar{u}})}+h.o.t. \label{eqweak1}
\end{align}
\begin{align}
  |(f_{u})_P(a,b)|^2\tau^{2d(P;f_u)}+h.o.t.&=|(f_{\bar{u}})_P(a,b)|^2\tau^{2d(P;f_{\bar{u}})}+h.o.t. \label{eqweak2}  \\
     |(f_{v})_P(a,b)|^2\tau^{2d(P;f_v)}+h.o.t.&= |(f_{\bar{v}})_P(a,b)|^2\tau^{2d(P;f_{\bar{v}})}+h.o.t. \label{eqweak3} \\
     f_P(a,b)\tau^{d(P;f)}+h.o.t.&=0\label{eqweak4},
\end{align}
where the higher order terms come from the higher order terms in $z(\tau)$ and terms of $f$ (or $f_x$, $x\in\{u,\bar{u},v,\bar{v}\}$) with greater radial degree with respect to $P$.
Since $k_1\geq \frac{p_1}{p_2} \geq k_N$, $\Delta(P)$ is not an extreme vertex. Therefore $f_P$ is neither $u$- and $\bar{u}$-semiholomorphic nor $v$- and $\bar{v}$-semiholomorphic. Hence, from Lemma~\ref{facefunc} we have  
\begin{equation*}
\begin{cases} 
&d(P;f_u)> d(P;f_{\bar{u}})=d(P;f)-p_1, \text{ if $f_P$ is $\bar{u}$-semiholomorphic} \\
 &d(P;f_{\bar{u}})> d(P;f_u)=d(P;f)-p_1, \text{ if $f_P$ is $u$-semiholomorphic} \\
 &d(P;f_{\bar{u}})=d(P;f_u)=d(P;f)-p_1, \text{$f_P$ depends on both $u$ and $\bar{u}$}
   \end{cases}
\end{equation*}
and 
\begin{equation*}
\begin{cases} 
&d(P;f_v)> d(f_{P;\bar{v}})=d(P;f)-p_2, \text{ if $f_P$ is $\bar{v}$-semiholomorphic} \\
 &d(P;f_{\bar{v}})> d(P;f_v)=d(P;f)-p_2, \text{ if $f_P$ is $v$-semiholomorphic} \\
 &d(P;f_{\bar{v}})=d(P;f_v)=d(P;f)-p_2, \text{ if $f_P$ depends on both $v$ and $\bar{v}$}
   \end{cases}
\end{equation*}
There are now 9 subcases to consider, corresponding to which of the variables $f_P$ depends on. The calculation is almost identical for these different cases, so that we only present one of them in detail. 

\vspace{0.3cm}

Suppose that $f_P$ is $\bar{u}$-semiholomorphic and depends on both $v$ and $\bar{v}$. Then $d(P;f_u)> d(P;f_{\bar{u}})=d(P;f)-p_1$ and $d(P;f_{\bar{v}})=d(P;f_v)=d(P;f)-p_2$. This means for example that the lowest degree of the left hand side in Eq.~\eqref{eqweak1} is greater than the lowest degree of the right hand side, which is $2d(P;f)-p_1-p_2$. By comparing coefficients of the the lowest degrees from Eq.~\eqref{eqweak1}-\eqref{eqweak4} we deduce 
\begin{align}\label{redusys}
 &(f_{v})_P(a,b) \cdot (\bar{f}_{\bar{u}})_P(a,b)=0 \nonumber\\
 &|(f_{\bar{u}})_P(a,b)|^2=0  \\
 &|(f_{v})_P(a,b)|^2= |(f_{\bar{v}})_P(a,b)|^2 \nonumber\\
 & f_P(a,b)=0 \nonumber
\end{align}
Applying~Lemma~\ref{facefunc} to this case we have $(f_{\bar{u}})_P= (f_P)_{\bar{u}}$, $(f_v)_P= (f_P)_v$ and $(f_{\bar{v}})_P= (f_P)_{\bar{v}}$ thus we have that the system \eqref{redusys} is equal to the system $s_{1,f_P}(a,b)=s_{2,f_P}(a,b)=s_{3,f_P}(a,b)=f_P(a,b)=0,$ which implies that $(a,b)$ belongs to $\Sigma_{f_P}\cap V_{f_P}\cap (\C^*)^{2}$, which is a contradiction of Condition~(ii) in Definition~\ref{Newtoncond}. The other 8 cases work analogously. 

\vspace{0.3cm}

Therefore, the curve $z(\tau)$ has a component that is equal to 0, or $k_1<\tfrac{p_1}{p_2}$ or $k_N>\tfrac{p_1}{p_2}$.   
\vspace{0.3cm}

If $k_1<\tfrac{p_1}{p_2}$, we write $f_{P_1}(z,\bar{z})=A(v,\bar{v})+B(v,\bar{v})u+C(v,\bar{v})\bar{u}+D(z,\bar{z})$ for some functions $A, B, C, D$, where the degree of $D$ with respect to $R=|u|=|\bar{u}|$ is greater than 1. Since $f$ has an inner Newton non-degenerate boundary, $f_{P_1}$ has no critical point in $\{(0,v):v\neq 0\}\cap V_{f_{P_1}}$, which is equivalent to the non-existence of solutions to the system given by:
\begin{equation}\label{almostnondeg}
\begin{cases} 
&B(v,\bar{v}) \overline{A_{\bar{v}}(v,\bar{v})}-\overline{C(v,\bar{v})}A_v(v,\bar{v})  =0, \\
 &|B(v,\bar{v})|^2-|C(v,\bar{v})|^2=0 \\
 &|A_v(v,\bar{v}) |^2-|A_{\bar{v}}(v,\bar{v})|^2=0\\ 
 & A(v,\bar{v})=0
   \end{cases}
\end{equation}
However, writing $f$ as the sum of $f_{P_1}$ and higher order terms (and likewise for $f_x$, $x\in\{u,\bar{u},v,\bar{v}\}$) and evaluating $s_{1,f}$, $s_{2,f}$, $s_{3,f}$ and $f$ on $z(\tau)$ results as in Eq.~\eqref{eqweak1}-\eqref{eqweak4} in a system of equations, whose lowest order terms in $\tau$ imply:
\begin{equation}\label{conv}
\begin{cases} 
&B(b,\bar{b}) \overline{A_{\bar{v}}(b,\bar{b})}-\overline{C(b,\bar{b})}A_v(b,\bar{b})  =0, \\
 &|B(b,\bar{b})|^2-|C(b,\bar{b})|^2=0 \\
 &|A_v(b,\bar{b}) |^2-|A_{\bar{v}}(b,\bar{b})|^2=0\\ 
 & A(b,\bar{b})=0
   \end{cases}
\end{equation}
if $f$ is $v$-convenient. The first equation for example reflects the coefficient of $\tau^{d(P;B)+d(P;A)-p_2}=\tau^{d(P;C)+d(P;A)-p_2}$ of the equation $s_{1,f}(z(\tau))=0$. We obtain the same system of equations from $s_{1,f}(z(\tau))=s_{2,f}(z(\tau))=s_{3,f}(z(\tau))=f(z(\tau))=0$ in the case where $z(\tau)=(0,b\tau^{p_2}+h.o.t.)$.

If $f$ is not $v$-convenient, i.e., $A(v,\bar{v})=0$, then Eq.~\eqref{almostnondeg} becomes simply $|B(v,\bar{v})|^2-|C(v,\bar{v})|^2=0$. In this case, the lowest degrees of Eq.~\eqref{eqweak1}-\eqref{eqweak4} imply
\begin{equation}\label{noconv}
\begin{cases} 
&B(b,\bar{b}) (\overline{B_{\bar{v}}(b,\bar{b})a+C_{\bar{v}}(b,\bar{b})\bar{a} })-(B_{v}(b,\bar{b})a+C_{v}(b,\bar{b})\bar{a})\overline{C(b,\bar{b})}  =0\\
 &|B(b,\bar{b})|^2-|C(b,\bar{b})|^2=0 \\
 &|B_{v}(b,\bar{b})a+C_{v}(b,\bar{b})\bar{a}|^2 -|B_{\bar{v}}(b,\bar{b})a+C_{\bar{v}}(b,\bar{b})\bar{a}|^2=0\\
 & B(b,\bar{b})a+C(b,\bar{b})\bar{a}=0.
   \end{cases}
\end{equation}
If $z(\tau)=(0,b\tau^{p_2}+h.o.t.)$, we obtain the same set of equations above with $a=0$.

Hence, in both cases, convenient and non-convenient, Eq.~\eqref{conv} and Eq.~\eqref{noconv} give a solution $(0,b), b\neq0$ of \eqref{almostnondeg}, which is a contradiction to Condition~(i). Analogously we can use Condition~(i) for $f_{P_N}$ to analyze the case of $z(\tau)=(a\tau^{p_1}+h.o.t.,0)$ or $k_N>\tfrac{p_1}{p_2}$. Therefore, $f$ has a weakly isolated singularity.  
\end{proof}

\section{SEMIHOLOMORPHIC PRINCIPAL PARTS}\label{section4}
In this section we are going to study the topology of inner non-degenerate mixed polynomials whose principal part is semiholomorphic. In this case, we can explicitly describe the link of the singularity via braids that correspond to the zeros of the individual face functions.

\begin{definition}
\label{def:russiandoll}
Let $(B_1,B_2,\ldots,B_N)$ be a sequence of geometric braids on $S_i$ strands, where $B_i$ is parametrized by
\begin{equation}
\bigcup_{j=1}^{S_i}\bigcup_{t\in[0,2\pi]}\{(u_{i,j}(t),t)\}\subset\mathbb{C}\times[0,2\pi],
\end{equation}
$u_{i,j}(t)\neq 0$ for all $j$, $t$ and $i\neq1$, i.e., $B_i$ with $i>1$ is affine. We define the braid $B(B_1,B_2,\ldots,B_N)$ to be the braid given by
\begin{equation}
\bigcup_{i=1}^N\bigcup_{j=1}^{S_i}\bigcup_{t\in[0,2\pi]}\{(\varepsilon^{k_i} u_{i,j}(t),t)\}\subset\mathbb{C}\times[0,2\pi],
\end{equation}
for some sufficiently small $\varepsilon>0$ and a sequence of strictly decreasing positive real numbers $k_i$, $i=1,2,\ldots,N$.

We denote by $B^o(B_1,B_2,\ldots,B_N)$ the link given by the union of the closed braid $B(B_1,B_2,\ldots,B_N)$ and its braid axis
\begin{equation}
    \bigcup_{\varphi\in[0,2\pi]}\{(\rme^{\rmi \varphi},\pi)\}\subset\mathbb{C}\times[0,2\pi].
\end{equation}
\end{definition}
Note that the braid type of $B(B_1,B_2,\ldots,B_N)$ does not depend on $\varepsilon$ as long as it is chosen sufficiently small and does not depend on the chosen sequence $k_i$ as long as it is strictly decreasing and positive. See Figure (\ref{fig:braids}) for an illustration of an example with $N=3$.

\begin{figure}\label{NestedBraid(1)}
    \centering
    \includegraphics[height=9cm]{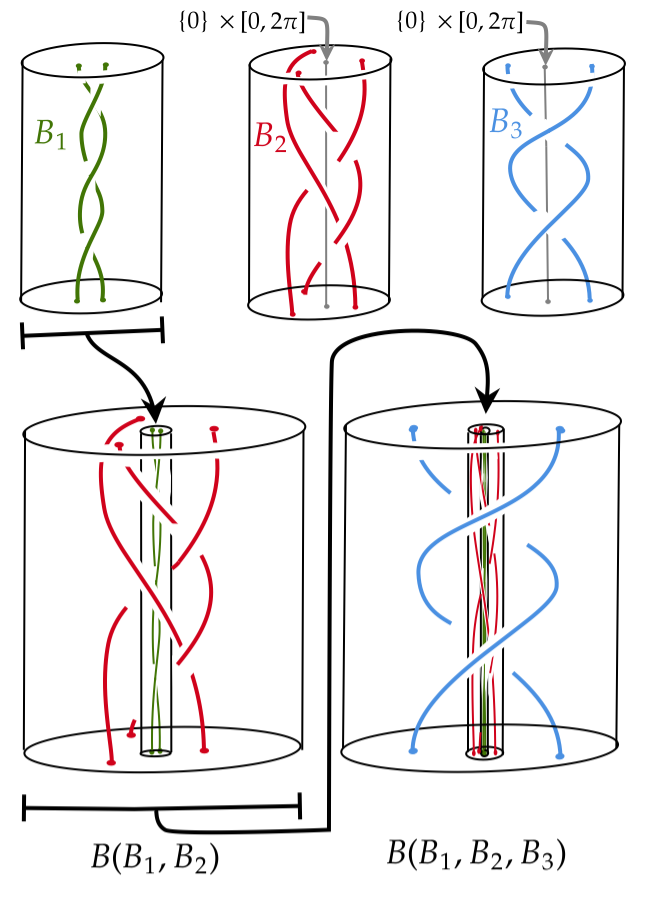}
    \caption{In the top row: a braid $B_{1}$ and two affine braids $B_{2},B_{3}$. In the bottom row: the braids $B(B_1,B_2)$ and $B(B_1,B_2,B_3)$.}
    \label{fig:braids}
\end{figure}

This construction clearly depends on the explicit parametrization of the geometric braids $B_i$. Even an overall translation of the form $u_{i,j}(t)\mapsto u_{i,j}(t)+c$ for some complex number $c$ and all $j\in\{1,2,\ldots,S_i\}$ in a given braid $B_i$ can change the topology of the resulting braid $B(B_1,B_2,\ldots,B_N)$. However, the braid isotopy class of $B(B_1,B_2,\ldots,B_N)$ and the link type of $B^o(B_1,B_2,\ldots,B_N)$ only depend on the braid isotopy class of $B_1$ and the affine braid isotopy classes of the $B_i$s with $i>1$.

Let $f:\mathbb{C}^2\to \mathbb{C}$ be a mixed polynomial and let $\mathcal{P=}\{P_1,P_2,\ldots,P_N\}$ be the sequence of positive weight vectors associated to the compact 1-faces of the Newton boundary of $f$. For each $P_i \in \mathcal{P}$ consider $f_{P_i}$ and $g_{i}$ as in the Eq.~\eqref{resc} where $k_i:=\tfrac{p_{i,1}}{p_{i,2}}$, $n_i$ is the smallest exponent of $r=|v|=|\bar{v}|$ in $f_{P_i}$ and $s_i$ is the largest exponent of $R:=|u|=|\bar{u}|$ in $f_{P_i}$.
\begin{lemma}\label{facedecom}
For every $i=1,2,\ldots,N$ the function $f$ admits the decomposition
\begin{equation}\label{decomp}
f(u,\bar{u},r\rme^{\rmi t},r\rme^{-\rmi t})= r^{k_i s_i+n_i}f_i\left(\frac{u}{r^{k_i}},\frac{\bar{u}}{r^{k_i}},r, t\right),
\end{equation}
where $f_{i}$ is a $r$-parameter deformation of $g_{i},$ that is, the difference $f_i(u,\bar{u},r,t)-g_i(u,\bar{u},\rme^{\rmi t})$ goes to $0,$ as $r\to 0.$ 
\end{lemma}
\begin{proof}

Fix $P_i\in \mathcal{P}$ and as in \eqref{resc} consider  $g_i:\mathbb{C}\times S^1\to\mathbb{C}$ so that 
\begin{equation}\label{decomp2}
f_{P_i}(u,\bar{u},r\rme^{\rmi t},r\rme^{-\rmi t})=r^{k_i s_i+n_i}g_i\left(\frac{u}{r^{k_i}},\frac{\bar{u}}{r^{k_i}},\rme^{\rmi t}\right).
\end{equation}

Let $n_i$ be the smallest exponent of $r\defeq|v|=|\bar{v}|$ in $f_{P_i}$ and $s_i$ the largest exponent of $R=|u|$ in $f_{P_i}$. Defining $f_i:\mathbb{C}\times \R_{\geq 0}\times [0,2\pi]\to\mathbb{C}$ via Eq.~\eqref{decomp} results in
\begin{equation}
f_i(w,\bar{w},r,t)=g_i\left(w,\bar{w},\rme^{\rmi t}\right)+\sum_{\nu+\mu \notin \Delta(P_i,f)}r^{-k_i s_i-n_i}M_{\nu,\mu}(r^{k_i}w,r^{k_i}\bar{w},r\rme^{\rmi t},r\rme^{-\rmi t})
\end{equation}
where $g_{i}$ comes from \eqref{decomp2}, and the sum on the right side runs over all monomials $M_{\nu,\mu}$ of $f$ with $rdeg_{P_i}(M_{\nu,\mu})>d(P_i;f)$.

It is important to remember that even though $r$ and $t$ are initially interpreted as the modulus and argument of the complex variable $v$, they are considered here as independent parameters; that is, even when $r=0$ the variable $t$ may still vary from 0 to $2\pi$.

\vspace{0.2cm}

In order to prove the property that $f_i(u,\bar{u},t)-g_i(u,\bar{u},\rme^{\rmi t})$ goes to 0 as $r$ goes to 0 we write $M_{\nu,\mu}(R e^{i\varphi},R e^{-i\varphi},r\rme^{\rmi t},r\rme^{-\rmi t})$ as $\phi(t,\varphi)R^a r^{b}$ for some function $\phi(t,\varphi)$. It follows from $rdeg_{P_i}(M_{\nu,\mu})>d(P_i;f)$ that $ak_i+b>k_i s_i+n_i$, which is equivalent to
\begin{equation}\label{restdec}
ak_i+b-k_i s_i-n_i>0.
\end{equation}
Note that the left side of Eq.~\eqref{restdec} is the degree of the variable $r$ of the monomial $$r^{-k_i s_i-n_i}M_{\nu,\mu}(r^{k_i}w,r^{k_i}\bar{w},r\rme^{\rmi t},r\rme^{-\rmi t}),$$ and hence it is precisely the condition for this monomial to be divisible by $r$, so that $f_i(u,\bar{u},0,t)=g_i(u,\bar{u},\rme^{\rmi t})$. Therefore, we can regard $f_i$ as a $r$-parameter deformation of $g_i$ whose difference $f_i(u,\bar{u},t)-g_i(u,\bar{u},\rme^{\rmi t}) \to 0$ as $r \to 0,$ for any value of $t\in[0,2\pi]$.
\end{proof}

\begin{lemma}\label{lem:vertex}
Let $f$ be a mixed polynomial, whose principal part is semiholomorphic and whose Newton boundary is inner non-degenerate. Then for every non-extreme vertex $\Delta$ and for the extreme vertex $\Delta=\Delta_{N+1}$ there are no zeros of $f_{\Delta}$ in $(\C^*)^2$.
\end{lemma}
\begin{proof}
Let $\Delta$ be a non-extreme vertex and suppose that $(u_*,v_*)\in(\C^*)^2$ is a zero of $f_{\Delta}$. Since the principal part of $f$ is semiholomorphic $f_{\Delta}$ is of the form $f_{\Delta}(u,r\rme^{\rmi t},r\rme^{-\rmi t})=u^nr^m\Phi(t)$ for some function $\Phi:S^1\to\mathbb{C}$. Since $(u_*,v_*)=(u_*,r_*\rme^{\rmi t_*})$ is a zero of $f_{\Delta}$, but also $u_*\neq 0$, $v_*\neq0$, we must have $\Phi(t_*)=0$. But then $\tfrac{\partial f_{\Delta}}{\partial u}(u_*,v_*)=\tfrac{\partial f_{\Delta}}{\partial r}(u_*,v_*)=0$ and hence $(u_*,v_*)\in(\C^*)^2$ is a critical point of $f_{\Delta}$ contradicting Condition (ii) in Definition~\ref{Newtoncond}.

For the extreme vertex $\Delta=\Delta_{N+1}$ consider that if $f$ is $u$-convenient, then $f_{\Delta}(u,v,\bar{v})=u^nc$ for some $c\in\C^*$, so that $f_{\Delta}$ has no zeros with $u\neq 0$. If $f$ is not $u$-convenient, then $\Delta$ has distance 1 from the horizontal axis and thus $f_{\Delta}(u,v,\bar{v})=u^n(ax+by)$ for some $a,b\in C$, not both equal to 0. If $(u_*,v_*)=(u_*,(x_*+\rmi y_*))\in(\C^*)^2$ is a zero of $f_{\Delta}$, then $a$ and $b$ are linearly dependent over $\mathbb{R}$. We compute $\tfrac{\partial f_{P_N}}{\partial u}(u,0,0)=0$, $\tfrac{\partial f_{P_N}}{\partial x}=u^na$ and $\tfrac{\partial f_{P_N}}{\partial y}=u^nb$. Since $a$ and $b$ are linearly dependent over $\mathbb{R}$, $u^na$ and $u^nb$ are linearly dependent over $\mathbb{R}$ and hence $(u,0)$ is a critical point of $f_{P_N}$ for all $u$. Since by assumption $f$ is not $u$-convenient, we also have $f_{P_N}(u,0)=0$, which contradicts Condition (i) in Definition~\ref{Newtoncond}.
\end{proof}
Note that for non-extreme vertices $\Delta$ of semiholomorphic principal parts the converse is also true: If there are no zeros of $f_{\Delta}$, then $f_{\Delta}$ is Newton non-degenerate.

We consider now semiholomorphic principal parts of inner non-degenerate mixed polynomials. Recall that  $\mathcal{P}=\{P_1,P_2,\ldots,P_N\}$ is the set of weight vectors associated to compact 1-faces of the Newton boundary. For every $i$ we have a corresponding $g_i$ defined from the face function $f_{P_i}$ via Eq.~\eqref{decomp2}.  For any $P_i=(p_{i,1},p_{i,2})\in \mathcal{P}$ consider $k_i=\tfrac{p_{i,1}}{p_{i,2}},$ $n_i$ and $s_i,$ as in the Lemma~\ref{facedecom} and let $m_i$ be the smallest exponent of $R\defeq|u|$ in $f_{P_i}$.

\begin{lemma}\label{lem:gibraid}
Let $f:\C^2\to\C$ be an inner non-degenerate mixed polynomial, whose principal part is semiholomorphic.  Then the zeros of $g_1(\cdot,\rme^{\rmi t})$ trace out a geometric braid on $s_1$ strands, which we denote by $B_1$, as $t$ varies from 0 to $2\pi$.
For every $i>1$ the non-zero roots of $g_i(\cdot,\rme^{\rmi t})$ trace out a geometric braid on $s_i-m_i$ strands, which we denote by $B_i$, as $t$ varies from 0 to $2\pi$. 
\end{lemma}
\begin{proof}
Lemma~\ref{lem:vertex} implies that the degree of $g_i(\cdot,\rme^{\rmi t})$, assigned to $f_{P_{i}}$, $i>1$, as in Eq.~\eqref{decomp2}, is equal to $s_i$ for all values of $t\in[0,2\pi]$ and the number of non-zero roots of $g_i(\cdot,\rme^{\rmi t})$ is equal to $s_i-m_i$ for all values of $t\in[0,2\pi]$. Likewise, the degree of $g_1(\cdot,\rme^{\rmi t})$ is equal to $s_1$ for all $t\in[0,2\pi]$.

We now have to show that all non-zero roots of $g_i(\cdot,\rme^{\rmi t})$, $i>1$, and all roots of $g_1(\cdot,\rme^{\rmi t})$ are simple for all $t\in[0,2\pi]$. This can be seen as follows. Suppose that $\tfrac{\partial g_i}{\partial u}(u_*,\rme^{\rmi t_*})=0$ for some $(u_*,t_*)\in\mathbb{C}^*\times[0,2\pi]$ with $g_i(u_*,\rme^{\rmi t_*},)=0$ (or equivalently $f_{P_i}(r^{k_i}u_*,r\rme^{\rmi t_*},r\rme^{-\rmi t_*})=0$ for all $r\in\mathbb{R}_{\geq 0}$). Then 
\begin{equation}
\tfrac{\partial f_{P_i}}{\partial u}(r^{k_i}u_*,r\rme^{\rmi t_*},r\rme^{-\rmi t_*})=\tfrac{\partial f_{P_i}}{\partial r}(r^{k_i}u_*,r\rme^{\rmi t_*},r\rme^{-\rmi t_*})=0,
\end{equation} 
and hence $Df_{P_i}$ does not have maximal rank at $(r^{k_i}u_*,r\rme^{\rmi t_*})$, which contradicts the assumption that $f$ is inner non-degenerate.

This means that for $i>1$ the non-zero complex roots $u_{i,\ell}(t)$, $\ell=1,2,\ldots,s_i-m_i$, of $g_i(\cdot,\rme^{\rmi t})$ trace out a geometric braid $B_i$ on $s_i-m_i$ strands, when $t$ varies from $0$ to $2\pi$.

The corresponding calculation for $g_1$ and a supposedly non-simple root $u_*\in\C$ of $g_1(\cdot,\rme^{\rmi t_*})$ for some $t_*\in[0,2\pi]$ is identical and leads to a contradiction as well. Hence the roots of $g_1(\cdot,\rme^{\rmi t})$ trace out a geometric braid $B_1$ on $s_1$ strands.
\end{proof}
Note that Lemma~\ref{lem:gibraid} implies that $f$ is true, since $s_i>m_i$ for all $i$. 

We now prove that the link of the weakly isolated singularity of a semiholomorphic boundary polynomial with an inner non-degenerate boundary can be constructed from the $B_is$ as described in Definition~\ref{def:russiandoll}.

\begin{prop}\label{prop:boundary} Let $f:\mathbb{C}^2\to \mathbb{C}$ be a $u$-semiholomorphic boundary polynomial with an inner non-degenerate boundary. Then $f$ has a weakly isolated singularity at the origin. Moreover, the link $L_f$ is isotopic to the closure of the braid $B(B_1,B_2,\ldots,B_{N})$ if $f$ is $u$-convenient. If $f$ is not $u$-convenient, then the link of the singularity is $B^o(B_1,B_2,\ldots,B_N)$. \end{prop}
\begin{proof}
The fact that $f$ has a weakly isolated singularity follows from the inner non-degeneracy of its Newton principal part by Proposition~\ref{weak-isolated}. Since $f$ is semiholomorphic, the exponents of $|u|$ coincide with the exponents of $u$. The decomposition~\eqref{decomp} of $f$ may be written as
\begin{equation*}
f(u,r\rme^{\rmi t},r\rme^{-\rmi t})= r^{k_i s_i+n_i}f_i\left(\frac{u}{r^{k_i}},r, t\right).
\end{equation*}
For $(u_*,r_*\rme^{\rmi t_*})\in(\C^*)^{2},$ we have that

\vspace{0.2cm}

$f_i(u_*,r_*,t_*)=0$ if and only if $f(r_*^{k_i} u_*,r_*\rme^{\rmi t_*},r_*\rme^{-\rmi t_*})=0.$ 

\vspace{0.2cm}




We already know that the non-zero complex roots $u_{i,\ell}(t)$, $\ell=1,2,\ldots,s_i-m_i$, $i>1$, of $g_i(\cdot,\rme^{\rmi t})$ trace out a geometric braid $B_i$ on $s_i-m_i$ strands and the roots $u_{1,\ell}$, $\ell=1,2,\ldots,s_1$, of $g_1(\cdot,\rme^{\rmi t})$ trace out $B_1$, when $t$ varies from $0$ to $2\pi$. Thus, we have $f_i(u_{i,\ell}(t),0,t)=0$ and these are the only zeros of $f_1$ at $r=0$ and the only zeros of $f_i$ at $r=0$ with non-vanishing $u$-coordinate. We may apply the implicit function theorem (or, use the continuity of the roots of polynomials in terms of their coefficients) to find disjoint tubular neighborhoods of the braids $B_i$ on $\mathbb{C}^*\times S^1$, on which the roots of $f_i$ remain simple and are parametrized by $r$ and $t$ as long as $r$ is chosen sufficiently small. In particular, for any fixed small enough value of $r$, the roots of $f_i$ form a braid $B_i(r)$ that is isotopic as a closed braid to $B_i$, where varying $r$ yields an explicit isotopy in $\C\times S^1$ starting from $B_i(0)=B_i$ at $r=0$.

\vspace{0.3cm}

Recall that if $B_i(r)$ is parametrized by $(u_{i,\ell}(r,t),t)\subset\C^*\times [0,2\pi]$, i.e., $f_i(u_{i,\ell}(r,t),r,t)=0$, then 
\begin{equation}
\label{eq:semiroots}
f(r^{k_i}u_{i,\ell}(r,t),r\rme^{\rmi t},r\rme^{-\rmi t})=0.
\end{equation} 
By continuity we have that if $r$ is sufficiently small, then 
\begin{equation}
(r^{k_i}u_{i,\ell}(r,t),r\rme^{\rmi t},r\rme^{-\rmi t})\neq (r^{k_j}u_{j,\ell'}(r,t),r\rme^{\rmi t},r\rme^{-\rmi t})
\end{equation} 
for all $i\neq j$ and $\ell=1,2,\ldots,s_i-m_i$, $\ell'=1,2\ldots,s_j-m_j.$ In addition, since $u_{i,\ell}(0,t)=u_{i,\ell}(t)\neq 0$, $\ell=1,2,\ldots,s_i-m_i$, we have by continuity $u_{i,\ell}(r,t)\neq 0$ for all $t$ and sufficiently small $r$. Note also that, even if there is a root $u_{1,\ell}(0,t)=0$, $\ell>s_1-m_1$, the arguments above hold on the vanishing root as well, because the inner non-degeneracy condition of $f$ on the positive weight vectors associated to the extreme vertex implies that  $u_{1,\ell}(t)=0$ is a simple root of $g_1(\cdot,\rme^{\rmi t})$.  

For any fixed value of $r$ and $t$, the polynomial $f$ is a complex polynomial in $u$ of degree $s_N$. Thus, for small values of $r$ we have already found $s_N$ roots of $f(\cdot,r\rme^{\rmi t},r\rme^{-\rmi t})$, $r>0$ via Eq.~\eqref{eq:semiroots}, so that there can be no further zeros of $f$ in a neighbourhood of the origin:
\begin{equation}
V_f\cap D_{\rho}^4\backslash\{(u,0):u\in\C\}=\bigcup_{i,\ell}\bigcup_{t\in[0,2\pi]}\bigcup_{r\in\mathbb{R}_{\geq0}}\{(r^{k_i}u_{i,\ell}(r,t),r\rme^{\rmi t})\in\mathbb{C}^2\}\cap D_{\rho}^4.
\end{equation}

If $f$ is $u$-convenient, then $f(u,0)=u^{s_N}$ (possibly times a constant), so that the origin is the only zero of $f$ with $v=0$.

Therefore, since the complex roots of $f(\cdot,r\rme^{\rmi t})$ are simple for sufficiently small $r>0$, it follows that the origin is a weakly isolated singularity of $f$ and by Definition~\ref{def:russiandoll} the corresponding link is the closure of $B(B_1,B_2,\ldots,B_{N}).$ The isotopy between the projection of $V_f\cap(\mathbb{C}\times \rho S^1)$ on $S_{\rho}^3$ and $V_f\cap S_{\rho}^3$ can be constructed analogously to that one in \cite[Theorem 1]{Dennis_Bode2019}. The number of strands $S_i$ in Definition~\ref{def:russiandoll} equals $s_i-m_i=s_i-s_{i-1}$ with $s_0=0$.

If $f$ is not $u$-convenient, the set $v=0$ is part of $V_f$. Since the Newton boundary of $f$ is inner non-degenerate, these zeros are regular points of $f$ and the origin is still a weakly isolated singularity of $f$. The link of the singularity is formed by the union of the closed braid $B(B_1,B_2,\ldots,B_{N})$ (with the isotopy between the projection of $V_f\cap(\mathbb{C}\times \rho S^1)$ on $S_{\rho}^3$ and $V_f\cap S_{\rho}^3$ as in the previous case) and $v=0$, which is the braid axis for $B(B_1,B_2,\ldots,B_{N})$. By definition this union is $B^o(B_1,B_2,\ldots,B_{N})$.
\end{proof}

The result above is concerned with boundary polynomials. We now prove that adding terms above the Newton boundary does not affect the result, even if the added terms themselves are not semiholomorphic.

\begin{prop}
\label{prop:semi}
Let $f:\mathbb{C}^2\to\mathbb{C}$ be a inner non-degenerate mixed polynomial, whose Newton principal part $f_{\Gamma}$ is semiholomorphic. Then $f$ has a weakly isolated singularity and the link $L_f$ is isotopic to $L_{f_{\Gamma}}$.
\end{prop}
\begin{proof}
Much of the proof is completely analogous to that of Proposition~\ref{prop:boundary}. The difference is that some terms above the Newton boundary of $f$ are mixed instead of semiholomorphic. 

But, again, the implicit function theorem implies that for small enough values of $r$ and in a tubular neighborhood of $B_i$ the zeros of $f_i$, as in Eq.~\eqref{decomp}, are parametrised by $r$ and $t$. Again, this shows that
\begin{equation}
f(r^{k_i}u_{i,\ell}(r,t),r^{k_i}\overline{u_{i,\ell}(r,t)},r\rme^{\rmi t},r\rme^{-\rmi t})=0.
\end{equation}
These roots form the closure of the braid $B(B_1,B_2,\ldots,B_N)$, which is $L_{f_{\Gamma}}$. By continuity the Jacobian of $f_i$ has full rank on its zeros $B_i(r)$ and hence the Jacobian of $f$ has full rank on its zeros forming $B(B_1,B_2,\ldots,B_N)$. Therefore, what is left to show is that these are the only zeros of $f$ in a neighborhood of the origin, except a component with $v=0$ in the case that $f$ is not $u$-convenient.

Since $f$ is not necessarily semiholomorphic, we cannot use the same argument as in the proof of Proposition~\ref{prop:boundary}. For simplicity we suppose that $f$ is $u$-convenient, otherwise we can divide by $r$ and the argument still holds for $f/r$. 

We now write $u=R\rme^{\rmi \varphi}$ and $\bar{u}=R\rme^{-\rmi\varphi}$ and consider $f$ as a polynomial in the variable $R$, keeping $\varphi$ as a parameter. However, we want to think of $R$ as a complex variable. Of course, eventually we are only interested in solutions with real values of $R$. Thus we write 
\begin{equation}
F_{\varphi,r,t}(R)\defeq f(R\rme^{\rmi \varphi},R\rme^{-\rmi \varphi},r \rme^{\rmi t},r \rme^{-\rmi t})
\end{equation}
and consider this family of polynomials as deformations of
\begin{equation}
F_{\varphi,0,t}(R)=a_{s_N}(R\rme^{\rmi \varphi})^{s_N}+\sum_{j_1,j_2}a_{j_1,j_2}R^{j_1+j_2}\rme^{\rmi(j_1-j_2)\varphi},
\end{equation}
where 
\begin{equation}
f(u,\bar{u},0,0)=a_{s_N}u^{s_N}+\sum_{j_1,j_2}a_{j_1,j_2}u^{j_1}\bar{u}^{j_2}.
\end{equation}
Note that $j_1+j_2>s_N$, since all these terms lie above the Newton boundary. If $f$ is $u$-convenient, $a_{S_N}\in\mathbb{C}^*$. If $f$ is not $u$-convenient, then in the corresponding equation for $f/r$ the coefficient depends on $t$, but by Lemma~\ref{lem:vertex} is never zero.

A quick remark is in order. Even though $R$, $\varphi$, $r$ and $t$ are originally defined as polar coordinates of complex variables, they are considered here as independent variables, i.e., even when $r=0$ we can vary $t$ from to $2\pi$ (and likewise for $R$ and $\varphi$). This variation has no effect on the resulting function, but it allows us to proceed as follows. 

Note that $R=0$ is a root of $F_{\varphi,0,t}$ of multiplicity $s_N$ for all values of $\varphi$ and $t$. Therefore there is a neighbourhood $U\subset \mathbb{C}$ of $R=0$ such that for any small enough value of $r$ and any value of $\varphi$ and $t$ the neighbourhood $U$ contains exactly $s_N$ roots of $F_{\varphi,r,t}$ when counted with multiplicity. This is a classical result on complex polynomials. A proof can be found (among others) in \cite{controots}.

We are only interested in real solutions $R$ of $F_{\varphi,r,t}$. The way that $F_{\varphi,r,t}$ is defined means that there is a relation between the roots at different values of the parameter $\varphi$ (while $r$ and $t$ are fixed). Namely, $R\in\mathbb{C}$ is a root of $F_{\varphi,r,t}$ if and only if $R\rme^{-\rmi \varphi}$ is a root of $F_{0,r,t}$. (It follows that if $R=0$ is a root for some value of $\varphi$, it is a root for all values of $\varphi$.) 

Now suppose that around the origin in $\mathbb{R}^4\cong\mathbb{C}^2$ the curves $(r^{k_i} u_{i,j}(r,t),r\rme^{\rmi t})$, with $r>0$ small, $t\in[0,2\pi]$, are not the only zeros of $f$ with $r>0$, i.e., there is an $\varepsilon>0$ such that for all $\rho<\varepsilon$ the curves
\begin{align}
&(D^4_{\rho}\backslash\{(u,0):u\in\C\})\cap\{(r^{k_i} u_{i,j}(r,t),r\rme^{\rmi t}):r\geq0,t\in[0,2\pi]\}\nonumber\\
&\subset (D^4_{\rho}\backslash\{(u,0):u\in\C\})\cap f^{-1}(0)
\end{align}
are a proper subset of $(D^4_{\rho}\backslash\{(u,0):u\in\C\})\cap f^{-1}(0)$, where $D^4_{\rho}$ denotes the 4-ball of radius $\rho$ around the origin.

Then for every $\rho<\varepsilon$ there exist an $r$, $t\in[0,2\pi]$ and more than $s_N$ complex numbers $U_j(r,t)$, $j=1,2,\ldots,s'$, $s'>s_N$, with $|U_j(r,t)|^2+r^2=\rho^2$, such that $f(U_j(r,t),r\rme^{\rmi t})=0$. (We know that 0 is an isolated root of $f(\cdot,0)$ if $f$ is $u$-convenient and so can exclude the case of $r=0$.)

But that would imply that for every $\rho$ there exist an $r\leq\rho$, $t\in[0,2\pi]$, $R_j(r,t)=|U_j(r,t)|\leq\rho$ and $\varphi_j(r,t)=\arg(U_j(r,t))$, $j=1,2,\ldots,s'$, with $F_{\varphi_j(r,t),r,t}(R_j(r,t))=0$. By choosing $\rho$ sufficiently small all $U_j(r,t)$ are contained in $U$.

This implies that $R_j(r,t)\rme^{-\rmi \varphi(r,t)}$, $j=1,2,\ldots,s'$, are more than $s_N$ roots of $F_{0,r,t}$ in $U$, which contradicts the statement above.

We have thus shown that 
\begin{align}
&(D^4_{\rho}\backslash\{(u,0):u\in\C\})\cap\{(r^{k_i} u_{i,j}(r,t),r\rme^{\rmi t}):r>0,t\in[0,2\pi]\}\nonumber\\
&=D^4_{\rho}\cap f^{-1}(0)\backslash\{(u,0):u\in\C\})
\end{align}
for small enough $\rho>0$.
The polynomial $f$ is $u$-convenient if and only if $f_{\Gamma}$ is $u$-convenient, so that the $v=0$ forms part of $V_f$ if and only if it is part of $V_{f_{\Gamma}}$. 

Since there are no critical points of $f$ on $\{(r^{k_i} u_{i,j}(r,t),r\rme^{\rmi t}):r>0,t\in[0,2\pi]\}$, the singularity at the origin is weakly isolated, whether $f$ is $u$-convenient or not.  Since the curves $\{(r^{k_i} u_{i,j}(r,t),r\rme^{\rmi t}):r>0,t\in[0,2\pi]\}$ parametrize the closed braid $B(B_1,B_2,\ldots,B_N)$, its closure is the link of the singularity if $f$ is $u$-convenient.

If $f$ is not $u$-convenient, the link of it singularity is the union of the closed braid $B(B_1,B_2,\ldots,B_N)$ and the braid axis given by $v=0$.

The link of the singularity of $f$ is thus equivalent to the link of the singularity of $f_{\Gamma}$ in either case.
\end{proof}




\section{LINKS OF INNER NON-DEGENERATE MIXED POLYNOMIALS}\label{section5}
Now we generalize the results from Section~\ref{section4} to general mixed polynomials with inner non-degenerate Newton boundary. Note that in this case the different $g_i$s are not necessarily semiholomorphic. Thus the number of zeros of $g_i(\cdot,\rme^{\rmi t})$ may vary with $t$, so that its zeros in general do not form a braid, but some link $L$ in $\mathbb{C}\times S^1$. The vanishing set of $g_i$ could also be the empty set. In this case we say that $L$ is the \textit{empty link}.

\begin{definition}
\label{def:russiandoll2}
Let $(L_1,L_2,\ldots,L_{N-1})$ be an ordered sequence of (possibly empty) links, where every non-empty $L_i$ is parametrized by
\begin{equation}
\bigcup_{j=1}^{M_i}\bigcup_{\tau\in[0,2\pi]}\{(u_{i,j}(\tau),\rme^{\rmi t_{i,j}(\tau)})\}\subset\mathbb{C}\times S^1,
\end{equation}
where $j=1,2,\ldots,M_i$ is indexing the $M_i$ components of $L_i$ and $z_{i,j}(\tau)\neq 0$ for all $i\neq 1$,$j$ and $\tau$. 
We define $[L_1,L_2,\ldots,L_{N-1}]$ to be the link given by
\begin{equation}
\bigcup_{i=1}^{N-1}\bigcup_{j=1}^{M_i}\bigcup_{\tau\in[0,2\pi]}\{(\varepsilon^{k_i} u_{i,j}(\tau),\rme^{\rmi t_{i,j}(\tau)})\}\subset\mathbb{C}\times S^1,
\end{equation}
for some sufficiently small $\varepsilon>0$ and a sequence of strictly decreasing positive real numbers $k_i$, $i=1,2,\ldots,N-1$.
\end{definition}
This construction is the analogue of Definition~\ref{def:russiandoll}, but in $\C\times S^1$ instead of $\C\times[0,2\pi]$. In particular, if every $L_i$ is a closed braid $B_i$, then $[L_1,L_2,\ldots,L_{N-1}]=B(B_1,B_2,\ldots,B_{N-1})$.

\begin{definition}\label{def:russiandoll3}
Let $K_1$ be a link in $\C\times S^1$ that is either empty or parametrized by
\begin{equation}
\bigcup_{j=1}^M\bigcup_{\tau\in[0,2\pi]}\{(u_{j}(\tau),\rme^{\rmi t_{j}(\tau)})\}
\end{equation}
and $K_2$ be a link in $S^1\times\C$ that is either empty or parametrized by
\begin{equation}
\bigcup_{j=1}^{M'}\bigcup_{\tau\in[0,2\pi]}\{(\rme^{\rmi \varphi_{j}(\tau)},v_{j}(\tau))\}.
\end{equation}
We define the link $\widetilde{K_i}$, $i=1,2$, in $S^3$, which is empty if $K_i$ is empty and otherwise parametrized by 
\begin{equation}
\bigcup_{j=1}^M\bigcup_{\tau\in[0,2\pi]}\left\{\left(\varepsilon u_{j}(\tau),\sqrt{1-\varepsilon^2|u_j(\tau)|^2}\rme^{\rmi t_{j}(\tau)}\right)\right\}
\end{equation}
and
\begin{equation}
\bigcup_{j=1}^{M'}\bigcup_{\tau\in[0,2\pi]}\left\{\left(\sqrt{1-\varepsilon^2|v_j(\tau)|^2}\rme^{\rmi \varphi_{j}(\tau)},\varepsilon v_{j}(\tau)\right)\right\}
\end{equation}
for some small $\varepsilon>0$, respectively.

Then we write $L(K_1,K_2)$ for the link in $S^3$ given by $\widetilde{K_1}\cup\widetilde{K_2}$.
\end{definition}

\begin{figure}[H]
    \centering
    \includegraphics[height=10cm]{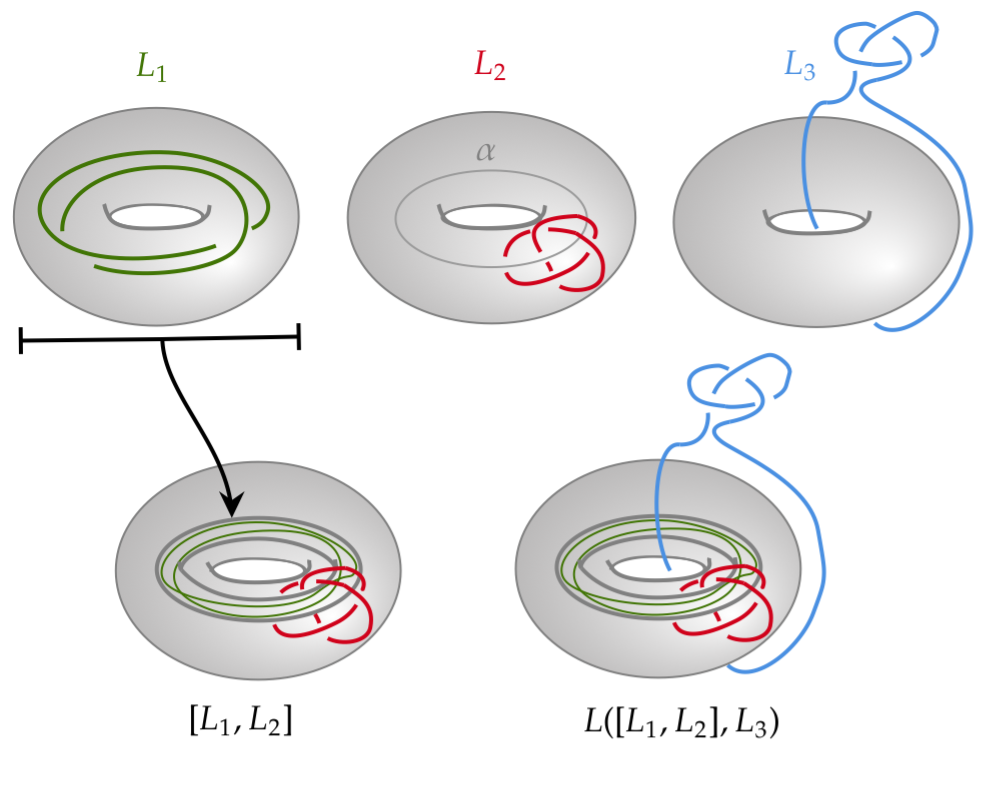}
    \caption{In the top row: links $L_1$ in $\C\times S^1$, $L_2$ in $(\C\backslash\{0\})\times S^1$, and $L_3$ in the complementary torus $S^1\times\C$. In the bottom row: the links $[L_1,L_2]$ and $L([L_1,L_2],L_3)$.}
    \label{fig:links}
\end{figure}

Figure~\ref{fig:links} displays the constructions defined in Definition~\ref{def:russiandoll2} and Definition~\ref{def:russiandoll3}. While they are given in terms of parametrizations of the links $L_i$, the link type of $L([L_1,L_2,\ldots,L_{N-1}],L_N)$ only depends on the isotopy class of $L_1$ in $\C\times S^1$, the isotopy classes of the $L_i$s with $i\in\{2,3,\ldots,N-1\}$ in $(\C\backslash\{0\})\times S^1$ and the isotopy class of $L_N$ in $S^1\times\C$.

The relation between the links in solid tori and links in $S^3$ can be understood as follows. The 3-sphere can be decomposed into two solid tori, which are glued together on their boundary. The cores of these solid tori are the unknots $\alpha:=\{(0,\rme^{\rmi t})\}_{t\in[0,2\pi]}$ and $\beta:=\{(\rme^{\rmi \varphi},0)\}_{\varphi\in[0,2\pi]}$, respectively. Definition~\ref{def:russiandoll3} thus interprets the solid torus $\C\times S^1$ containing $K_1$ as the solid torus in $S^3$ with core $\alpha$. Likewise the solid torus $S^1\times\C$ containing $K_2$ is identified with the complementary solid torus, whose core is $\beta$. Figure~\ref{fig:links} shows the corresponding links projected into $\mathbb{R}^3$, so that the solid torus $\C\times S^1$ is identified with a tubular neighbourhood of the unit circle in the $xy-$plane. Its complementary solid torus $S^1\times\C$ in $S^3$ is then obviously its complementary solid torus in $\mathbb{R}^3$ including the point at infinity. The image of $\beta$ under the stereographic projection map is the $z$-axis.

There are very close relations between Definitions~\ref{def:russiandoll2} and~\ref{def:russiandoll3} and Definition~\ref{def:russiandoll}. If $K_2$ does not intersect the core of its solid torus, i.e., $v_j(\tau)\neq 0$ for all $\tau\in[0,2\pi]$ and all $j=1,2,\ldots,M'$, then $K_2$ can also be interpreted as a link in $\C\times S^1$ via
\begin{equation}
\bigcup_{j=1}^{M'}\bigcup_{\tau\in[0,2\pi]}\{(u_j(\tau),\rme^{\rmi t_j(\tau)})\},
\end{equation}
where $|u_j(\tau)|=\sqrt{1-\varepsilon|v_j(\tau)|^2}$, $\arg(u_j(\tau))=\varphi_j(\tau)$ and $t_j(\tau)=\arg(v_j(\tau))$. In this case $L(K_1,K_2)$ is equal to $[K_1,K_2]$ interpreted as a link in $\C\times S^1\subset S^3$.

In particular, if $K_2$ does not intersect the core of its solid torus $S^1\times \C$ and forms a closed geometric braid $B_N$ in $\C\times S^1$, then $L([B_1,B_2,\ldots,B_{N-1}],B_N)=B(B_1,B_2,\ldots,B_{N-1},B_N)=[B_1,B_2,\ldots,B_{N-1},B_N]$ for any sequence of closed geometric braids $B_1,B_2,\ldots,B_{N-1}$.

If $K_2$ contains a component that is equal to the core $(\rme^{\rmi \varphi},0)$ and the remaining components of $K_2$ (interpreted as a link in $\C\times S^1$) form a closed geometric braid $B_N$, we have $L([B_1,B_2,\ldots,B_{N-1}],K_2)=B^o(B_1,B_2,\ldots,B_{N-1},B_N)$.

Definition~\ref{def:russiandoll2} and Definition~\ref{def:russiandoll3} allow us to construct links from a sequence of links in a way that generalizes the nested behaviour of the braids in Definition~\ref{def:russiandoll}. However, they allow us to deal with links in solid tori that are not closed braids and with links $K_2\subset S^1\times \C$ that intersect the core of the solid torus.

As stated in Theorem~\ref{thm:main} the links of the singularities of mixed polynomials with a certain specified behavior on its Newton boundary take the form $L([L_1,L_2,\ldots,L_{N-1}],L_N)$. Note that we could define an analogue of the construction in Definition~\ref{def:russiandoll2} for links in the solid torus $S^1\times \C$, which produces for any sequence of links $L_1,L_2,\ldots,L_N$ in $S^1\times \C$ a link $[L_1,L_2,\ldots,L_N]$ in $S^1\times \C$. Then the link $L([L_1,L_2,\ldots,L_{N-1}],L_N)$ could be equivalently represented by $L([L_1,L_2,\ldots,L_k],[L_N,L_{N-1},\ldots,L_{k+1}])$ for any $k=2,3,\ldots,N-1$. Note the order in the second bracket is reversed. This is because Definition~\ref{def:russiandoll2} arranges the links $L_i$ such that the $u$-coordinates of $L_i$ are strictly smaller than those of $L_j$ for all $j>i$ and a smaller $u$-coordinate corresponds to a larger $v$-coordinate in $S^3$.



In order to prove an analogue of Proposition~\ref{prop:main} for mixed polynomials we need an additional assumption.
\begin{definition}\label{def:blabla}
We say that a mixed polynomial $f:\mathbb{C}^2\to\mathbb{C}$ is \textbf{nice} if for every non-extreme vertex $\Delta$ of $f$,  $V_{f_{\Delta}}\cap (\C^*)^{2}=\emptyset$.
\end{definition}
Note that $f$ in Proposition~\ref{prop:boundary} is nice, since $f$ is nice in a semiholomorphic vertex if and only if $f$ is Newton non-degenerate for this vertex (cf. Lemma~\ref{lem:vertex}).

Let $\Delta$ be a non-extreme vertex and consider the weight vector $P_i$ such that $\Delta=\Delta(P_{i-1})\cap\Delta(P_i)$. By definition of $m_i$ the polynomials $g_i$ and $f_{\Delta}=R^{m_i}\Phi(\rme^{\rmi \varphi},\rme^{\rmi t})$ are both divisible by $R^{m_i}$. We may thus consider the function $\tfrac{g_i}{R^{m_i}}$ as a function of $R$, $\varphi$ and $t$. Again, we may vary $\varphi$ from 0 to $2\pi$ even when $R=0$. 

Note that $\tfrac{g_i}{R^{m_i}}(0,\rme^{\rmi \varphi},\rme^{\rmi t})=\Phi(\rme^{\rmi \varphi},\rme^{\rmi t})$. The condition of being nice guarantees that $\Phi$ never vanishes and thus we have for all values of $t$ and $\varphi$ that $\tfrac{g_i}{R^{m_i}}(0,\rme^{\rmi \varphi},\rme^{\rmi t})\neq 0$. Therefore the vanishing set of $g_i$ splits into two sets of connected component: The set $u=0$ with multiplicity $m_i$ and another set of components $L_i\subset\C\times S^1$ disjoint from $\{(0,\rme^{\rmi t})\}_{t\in[0,2\pi]}$.

As we have seen in the previous section this condition becomes simpler if $g_i$ is semiholomorphic. It means that $\tfrac{g_i}{u^{m_i}}(0,0,\rme^{\rmi t})\neq0$. (Note that in this case $g_i$ is divisible by $u^{m_i}$.) Furthermore, the degree of $g_i$ with respect to $u$ does not depend on $t$, so that the roots of $g_i(\cdot,\rme^{\rmi t})$ trace out a braid as $t$ varies from 0 to $2\pi$.   

Define $L_1$ as the link in $\C\times S^1$ formed by the zeros of the function $g_{1}$ associated to $f_{P_{1}},$ as in Eq.~\eqref{decomp2} and $L_N$ be the link in $S^1\times\C$ formed by the zeros of the function $\widehat{g}_N:S^1\times\C\to\C$ defined by
\begin{equation}
f_{P_N}(u,\bar{u},v,\bar{v})=f_{P_N}(R\rme^{\rmi \varphi},R\rme^{-\rmi \varphi},v,\bar{v})=R^{\hat{s}_i/k_N+m_N}\widehat{g}_N(\rme^{\rmi \varphi},v,\bar{v}), 
\end{equation}
where $\hat{s}_i$ is the degree of $f_{P_N}$ with respect to $r=|v|=|\bar{v}|$.

Define $L_{i},\ i=2,3,\ldots, N-1,$ as the link formed by the zeros of the function $g_i$ with non-zero $u$-coordinate as above. The fact that every $L_i$, $i=1,2,\ldots,N$ is a link rather than a set of curves with (self)intersections follows from the same line of reasoning as the proof of Lemma~\ref{lem:gibraid}. If there were some critical point of $g_i$ on $g_i^{-1}(0)$, then there would exist a corresponding line of critical points of $f_{P_i}$ in $V_{f_{P_i}}\cap(\C^*)^2$ contradicting the assumption that $f$ is inner non-degenerate. 

\begin{lemma}
\label{thm}
Let $f:\mathbb{C}^2\to \mathbb{C}$ be a mixed polynomial with a nice inner non-degenerate boundary. Then the link of the singularity $L_f$ contains the link $L([L_1,L_2,\ldots,L_{N-1}],L_{N})$ as a sublink.
\end{lemma}

\begin{proof}
Again the proof follows the ideas from Proposition~\ref{prop:boundary}. We consider $f_i:\mathbb{C}\times\mathbb{R}_{\geq 0}\times[0,2\pi]$ as in Eq.~\eqref{decomp} and 
again we have $f_i(u,\bar{u},0,t)=g_i(u,t)$ and the non-degeneracy of $f_{P_i}$ implies that we can evoke the implicit function theorem. However, since $g_i$ is not necessarily holomorphic with respect to $u$ anymore, this requires a bit of additional explanation. We know that the real Jacobian $Df_{P_i}$ of $f_{P_i}$ has full rank at any zero of $f_{P_i}$ in $(\C^*)^{2}$. We can choose local coordinates $(x,y,\tau,r)$ such that $\partial_\tau$ is the direction along the zeros $L_i$ of $g_i$, $i=1,2,\ldots,N-1$. For $i=N$ the local coordinates may take the form $(\tau,R,x,y)$. The following arguments apply to $i=N$ when the roles of $r$ and $R$ are exchanged. Since $f_{P_i}$ is radially weighted homogeneous and $g_i$, $i=1,2,\ldots,N-1$, does not depend on $r$, i.e., $f_{P_i}$ depends on $r$ only via scaling, the matrix
\begin{equation}
\begin{pmatrix}
\frac{\partial \text{Re}(g_i)}{\partial x} & \frac{\partial \text{Re}(g_i)}{\partial y}\\
\frac{\partial \text{Im}(g_i)}{\partial x} & \frac{\partial \text{Im}(g_i)}{\partial y}
\end{pmatrix}
\end{equation}
has full rank.

Hence, by the implicit function theorem, as long as $r$ is chosen small enough, the zeros of $g_i$ are parametrized by $r$ and $\tau$, i.e. $(x_{i,\ell}(r,\tau),y_{i,\ell}(r,\tau))$, where $\ell=1,2,\ldots,M_i$, now enumerates the connected components of the zeros of $g_i$. In particular, varying $r$ yields an explicit isotopy in $\C\times S^1$ between $L_i$, the zeros of $g_i$, and $L_i'(r)$, the zeros of $f_i$ close to $L_i$ at a fixed value of $r$, with $L_i'(0)=L_i$. We can rewrite the parametrization of the links in terms of our usual coordinates and obtain $(u_{i,\ell}(r,\tau),r\rme^{\rmi t_{i,\ell}(r,\tau)})$.

Again we obtain the cones of the $L_i$s as zeros of $f$, where the scaling of the $u$-coordinate (the ``sharpness'' of the cone) for each $i$ depends on the weight vector $P_i$ via $k_i=\tfrac{p_{i,1}}{p_{i,2}}$. To be precise,
\begin{equation}
f(r^{k_i}u_{i,\ell}(r,\tau),r^{k_i}\overline{u_{i,\ell}(r,\tau)},r\rme^{\rmi t_{i,\ell}(r,\tau)},r\rme^{-\rmi t_{i,\ell}(r,\tau)})=0
\end{equation}
for all $i=1,2,\ldots,N-1$, $\ell=1,2,\ldots,M_i$, $\tau \in S^1$ and small enough $r\geq0$. Likewise, for $i=N$ we obtain 
\begin{equation}
f(R\rme^{\rmi \varphi_j(R,\tau)},R\rme^{-\rmi \varphi_j(R,\tau)},R^{1/k_N}v_j(R,\tau),R^{1/k_N}\overline{v_j(R,\tau)})=0,
\end{equation}
for all small enough $R\geq 0$, with $j=1,2,\ldots,M'$ going over the $M'$ components of $L_N$.

Therefore $L([L_1,L_2,\ldots,L_{N-1}],L_N)$ is a sublink of the link of the singularity of $f$.
\end{proof}
Note that in contrast to Proposition~\ref{prop:main} we do not need to consider two different cases. The construction in Definition~\ref{def:russiandoll3} was devised in such a way that if $f$ is not $u$-convenient, $L_N$ has a component with $v=0$, but it does not affect the construction. It is simply the analogue of $B_1$ in Proposition~\ref{prop:main} having a component with $u=0$. We would like to emphasize again that $g_i(u,\rme^{\rmi t})$ or $\widehat{g_N}$ could potentially be non-vanishing for all $u\in\mathbb{C}$, $t\in[0,2\pi]$, in which case $L_i$ is the empty link.



\begin{teo}
\label{thmmixed}
Let $f:\mathbb{C}^2\to\mathbb{C}$ and $L_i$, $i=1,2,\ldots,N$, be as in Lemma~\ref{thm}. Then the link of the singularity is equal to $L([L_1,L_2,\ldots,L_{N-1}],L_N)$.
\end{teo}
\begin{proof}
Take $\rho>0$ sufficiently small and note that the set $V_f\cap D_{\rho}^4\cap(\C^*)^2$ is a semi-algebraic set. Therefore any of its connected components is a semi-algebraic set. We can thus apply the curve selection lemma to each of its connected components.

Suppose now that the link of the singularity has a component $K$ besides $L([L_1,L_2,\ldots,L_{N-1}],L_N)$ and different from the core $\{(0,\rme^{\rmi t})\}_{t\in[0,2\pi]}$ and the core $\{(\rme^{\rmi \varphi},0)\}_{\varphi\in[0,2\pi]}$. Then by the curve selection lemma there is a real-analytic path $\gamma:[0,\varepsilon)\to(\mathbb{C^*})^2$ with $\gamma(0)=O$ and $\gamma(\chi)$ in the connected component of $V_f\cap D_{\rho}^4\cap(\C^*)^2$ corresponding to $K$ for all $\chi\in(0,\varepsilon)$. We write $u_*(\chi)$ and $v_*(\chi)$ for the $u$- and $v$-coordinate of this path. It follows that $u_*(\chi)$ and $v_*(\chi)$ are analytic functions of $\chi$.


Since the curve lies in $(\C^*)^2\cup\{O\}$, neither of the two power series vanishes completely. Let $k$ be the quotient of the orders of these power series, i.e., if $u_*(\chi)=a_{j}\chi^{j}+h.o.t.$ and $v_*(\chi)=b_{j'}\chi^{j'}+h.o.t.$, then $k=j/j'$.
Note that this implies that $$\lim_{\chi\to 0}u_*(\chi)/|v_*(\chi)|^{k}=\tfrac{a_{j}}{|b_{j’}|^k}$$ is a finite, well-defined value.

If $k=k_i$ for some 1-face $P_i$ of the Newton boundary, then $\left(\tfrac{u_*(\chi)}{|v_*(\chi)|^{k_i}},v_*(\chi)\right)$ is a root of $f_i$ for every $\chi\in[0,\varepsilon)$. In particular, $\lim_{\chi\to 0}|u_*(\chi)/|v_*(\chi)|^{k_i}=\tfrac{a_{j}}{|b_{j’}|^{k_i}}$ is a  root of $g_i$ with non-zero $u$-coordinate. Therefore, it lies on $L_i$ and the path $(u_*(\chi),v_*(\chi))$ lies on the component of $V_f\cap D_{\rho}^4\cap(\C^*)^2$ corresponding to $L_i$, which is a contradiction.

Therefore, $k\neq k_i$ for any 1-face $P_i$ of the Newton boundary of $f$. In this case, it corresponds to the weight vector of some vertex $\Delta$ of the Newton boundary, say with coordinates $(m,n)$. Assume that $k>k_N$. Then we may define $f_{(m,n)}$ via
\begin{equation}
    f(u,\bar{u},r\rme^{\rmi t},r\rme^{-\rmi t})=r^{km+n}f_{(m,n)}\left(\frac{u}{r^k},\frac{\bar{u}}{r^k},r,t\right).
\end{equation}
Note that we have (like for the $f_i$s defined for 1-faces) $r^{km+n}f_{(m,n)}\left(\frac{u}{r^k},\frac{\bar{u}}{r^k},0,t\right)=f_{\Delta}(u,\bar{u},v,\bar{v})$.

The curve $\left(\tfrac{u_*(\chi)}{|v_*(\chi)|^{k}},\tfrac{\overline{u_*(\chi)}}{|v_*(\chi)|^{k}},|v_*(\chi)|,\arg(v_*(\chi))\right)$ is in $V_{f_{(m,n)}}$ for all sufficiently small $\chi>0$. It follows that the limit value $$\lim_{\chi\to 0}u_*(\chi)|/|v_*(\chi)|^{k}=\tfrac{a_{j}}{|b_{j’}|^k}$$ is the non-zero $u$-coordinate of a root of $f_{(m,n)}(u,\bar{u},0,t)$. Since  $f_\Delta(u,\bar{u},v,\bar{v})=r^{km+n} f_{(m,n)}\left(\frac{u}{r^k},\frac{\bar{u}}{r^k},0,t\right)$, this implies that there exists a corresponding root of $f_{\Delta}$ in $(\C^*)^{2}$, which contradicts the assumption that $f$ has a nice Newton boundary.

If $k<k_N$, the vertex $\Delta$ is the extreme vertex on $\Delta(P_N)$ and we define $f_{(m,n)}$ accordingly:
\begin{equation}
f(R\rme^{\rmi \varphi},R\rme^{-\rmi\varphi},v,\bar{v})=R^{n/k+m}f_{(m,n)}\left(R,\varphi,\frac{v}{R^{1/k}},\frac{\bar{v}}{R^{1/k}}\right).
\end{equation}
Analogously to the previous case, the limit value $$\lim_{\chi\to 0}v_*(\chi)|/|u_*(\chi)|^{1/k}=\tfrac{b_{j'}}{|a_{j}|^{1/k}}$$ would constitute a root of $f_{(m,n)}(0,\varphi,v,\bar{v})$, leading to a root of $f_{\Delta}$ in $(\C^*)^2$ and thereby contradicting the niceness assumption. Therefore $V_f\cap D_{\rho}^4\cap(\C^*)^2$ has no connected components besides the different $L_i$.

Suppose there is another component of $V_f\cap D_{\rho}^4\backslash\{O\}$ that is contained in $\{(u,0):u\in\mathbb{C}\}$, i.e., $v=0$. Then since all points of $V_f\cap D_{\rho}^4\backslash\{O\}$ are regular, this extra component is a 2-dimensional manifold. In other words, for $v=0$ the roots of $f(u,\bar{u},0,0)$ are a 2-dimensional subset of $\mathbb{C}$. But that means that $f(u,\bar{u},0,0)=0$ for all $u\in\mathbb{C}$. Thus $f$ is not $u$-convenient and $\{(u,0):u\in\mathbb{C}\}$ is a part of $L_N$.

Similarly, if there is another component with $u=0$, then $f$ is not $v$-convenient and the corresponding component is a component of $L_1$. Therefore, the link of the singularity is $L([L_1,L_2,\ldots,L_{N-1}],L_N)$.
\end{proof}

\begin{corolario}\label{corcomponents}
Let $f:\mathbb{C}^2\to\mathbb{C}$ be a mixed polynomial with a nice inner non-degenerate boundary. Then the number of components of the link of its weakly isolated singularity is equal to the sum of the number of components of $L_i$.
In particular, if a knot $K$ is the link of a weakly isolated singularity of a true mixed polynomial $f$ with a nice Newton inner non-degenerate boundary, then the Newton boundary of $f$ has exactly one compact 1-face. Therefore, there is a radially weighted homogeneous polynomial with a weakly isolated singularity and $K$ as the link of that singularity.
\end{corolario}

Compare this with Theorem 43 of Oka \cite{Oka2010}, which calculates the number of link components as the sum of the number of components of $L_i$ (summed over $i$) under the assumption that all vertices of $\Gamma(f)$ are simple.

Another application of our theorems is analogous to a result by Eyral and Oka \cite{EyralOka2018} for the class of non-convenient strongly non-degenerate mixed polynomial with isolated singularity.
\begin{corolario}\label{corbeconv}
Let $f$ be a non-convenient mixed polynomial, whose boundary is nice and Newton inner non-degenerate, and let $M_1(v,\bar{v})$ and $M_2(u,\bar{u})$ be mixed polynomials satisfying that the smallest exponents of the variable $|v|$ in $M_1$ and of the variable $|u|$ in $M_2$ are greater than $d(P_1;f)/p_{1,2}$ and $d(P_{N};f)/p_{N,1}$, respectively. Then $f+M_1+M_2$ is a convenient weak realization of $L_{f}$.  
\end{corolario}
\begin{proof}
Suppose that $f$ is not $v$-convenient. Then $L_0:=S^3_{\rho} \cap \{u=0\}$ is a link component of $L_1\subset L_f$. Take $F=f+M_1$ , where $M_1$ is a sum of monomials which satisfy that the exponent in variable $|v|$ is greater than  $d(P_1;f)/p_{1,2}$. Then $\mathcal{P}(F)= \{P_0\}\cup \mathcal{P}(f)$ for some weight vector $P_0$ with $k_0>k_i$ for all $i=1,2,\dots,N$. Note that there exists a $d>0$ and functions $\phi_1$, $\phi_2$ satisfying $$F_{P_0}(u,\bar{u},re^{it},re^{-it})=r^d(\phi_1(e^{it},e^{-it})u+\phi_2(e^{it},e^{-it})\bar{u}+r^{-d}(M_1)_{P_0}(re^{it},re^{it})),$$ where $\lim_{r\to 0} r^{-d}(M_1)_{P_0}(re^{it},re^{it}))=0$. By Condition~(i) of Definition~\ref{Newtoncond} we have that $(u=0,\rme^{\rmi t})$ is a regular point of $\phi_1(e^{it},e^{-it})u+ \phi_2(e^{it},e^{-it})\bar{u}$. Thus we can apply the implicit function theorem and the same arguments as in the proof of Theorem~\ref{thm:main} to show that the link associated to  $F_{P_0}$ is isotopic to $L_0$. In fact, this isotopy does not leave a tubular neighbourhood of $L_0$. The other components of $L_{f+M_1+M_2}$ are isotopic to those of $L_f$, since for every $P_i$, $i=1,2,\ldots,N$, the polynomial $M_1$ lies above $\Delta(P_i,f)$ and hence adding $M_1$ does not result in a change of the topology of $L_i$ nor of the way that it is linked with the other components. Following the same procedure if $f$ is not $u$-convenient we obtain that $f(u,\bar{u},v,\bar{v})+M_1(v,\bar{v})+M_2(u,\bar{u})$ is a convenient mixed polynomial and weak realization of $L_f$.     
\end{proof}

\section{STRONGLY INNER NON-DEGENERATE MIXED POLYNOMIALS}\label{section6}

Mixed polynomials with isolated singularities are of particular interest, since they are much rarer than weakly isolated singularities and harder to construct. In this section we show that under a strong Newton inner non-degeneracy condition of the boundary we get an isolated singularity at the origin.
\begin{definition}\label{def-strong}
We say that $f$ has a strongly Newton inner non-degenerate boundary if both of the following conditions hold:
\begin{itemize}
    \item[i.] the face functions $f_{P_1}$ and $f_{P_N}$ have no critical points in $ \C^2\setminus \{v=0\}$ and $\C^2\setminus \{u=0\}$, respectively.
\item[ii.] on each 1-face and non-extreme vertex $\Delta$, the face function $f_\Delta$ has no critical points in $(\C^*)^{2}$.
\end{itemize}
\end{definition}

Following the same idea of the proof of Proposition~\ref{weak-isolated} we get. 
\begin{prop}\label{strong-isolated}
Let $f:\C^2 \to \C$ be a mixed polynomial with a strongly Newton inner non-degenerate boundary. Then $f$ has an isolated singularity at the origin.   
\end{prop}
\begin{proof}
The proof is almost identical to that of Proposition~\ref{weak-isolated}. Assuming that the singularity is not isolated results via the curve selection lemma in a real-analytic curve $z(\tau)$ of critical points starting at the origin. Evaluating the defining equations $s_{1,f}$, $s_{2,f}$ and $s_{3,f}$ of $\Sigma_f$ on this curve yields a system of equations as in Eq.~\eqref{eqweak1}-\eqref{eqweak3} in terms of power series of $\tau$. Comparing lowest order terms yields a solution to the equations that define the strong inner non-degeneracy of $f_{P}$, where $P=(p_1,p_2)$ is the weight vector defined from the lowest order terms of $z(\tau)$ as in Proposition~\ref{weak-isolated}. Again there are the same cases to consider as in the proof of Proposition~\ref{weak-isolated}: $k_1\geq \tfrac{p_1}{p_2}\geq k_N$, $\tfrac{p_1}{p_2}< k_N$, $\tfrac{p_1}{p_2}> k_1$, and the vanishing of one of the complex coordinates of $z(\tau)$. In all of these cases, the arguments from Proposition~\ref{weak-isolated} remain unchanged, but since $z(\tau)$ is no longer assumed to be in $V_f$, it proves by contradiction that the origin is an isolated singularity. 
\end{proof}

As in the case of inner non-degeneracy our notion of strong inner non-degeneracy generalizes Oka's convenient strongly non-degenerate functions.
\begin{prop}
Let $f:\C^2\to\C$ be a convenient mixed function that is strongly non-degenerate in Oka's sense. Then $f$ has a strongly Newton inner non-degenerate boundary.
\end{prop}
\begin{proof}
As in the proof of Proposition~\ref{lem:inner} we only have to show that a critical point of $f_{\Delta}$ in $(\C^*)^2$ for $\Delta$ an extreme vertex of $f$ leads to a contradiction with Condition (i) in the definition of strong inner non-degeneracy. Assume that $\Delta$ is the extreme vertex lying on $\Delta(P_1)$. Since $f$ is convenient, $f_{\Delta}$ takes the form $f_{\Delta}(u,\bar{u},v,\bar{v})=A(v,\bar{v})$ for some function $A:\C\to\C$. Thus $s_{1,f_{\Delta}}=s_{2,f_{\Delta}}=0$ and a critical point of $f_{\Delta}$ in $(\C^*)^2$ corresponds to a non-zero root of $s_{3,f_{\Delta}}=|A_v(v,\bar{v})|^2-|A_{\bar{v}}(v,\bar{v})|^2$.

Condition (i) in Definition~\ref{def-strong} is equivalent to the non-existence of solutions $\{(0,v):v\neq0\}$ to the first three equations in Eq.~(\ref{almostnondeg}). In particular at $u=0$, the third equation is $s_{3,f_{P_1}}|_{u=0}=s_{3,f_{\Delta}}$. Hence if $f$ is strongly non-degenerate, there are no non-zero roots of $s_{3,f_{\Delta}}$ and thus Condition (i) is satisfied for $f_{P_1}$. The arguments for $f_{P_N}$ are analogous.
\end{proof}

\section{REAL ALGEBRAIC LINKS}\label{section7}
We prove that for certain braids of the form $B(B_1,B_2,\ldots,B_N)$ we can construct semiholomorphic polynomials that are strongly inner non-degenerate with the closure of $B(B_1,B_2,\ldots,B_N)$ as the link of a singularity. By Proposition~\ref{strong-isolated} the singularity at the origin is not only weakly isolated, but isolated and hence the closure of $B(B_1,B_2,\ldots,B_N)$ is real algebraic.

In order to characterize the sequences of braids $(B_1,B_2,\ldots,B_N)$ for which this procedure works, we need to generalize the concept of P-fiberedness from \cite{bodesat}. Recall that a braid is P-fibered if the corresponding loop in the space of polynomials $g(\cdot,\rme^{\rmi t})$, with $g^{-1}(0)=B$, defines a fibration via $\arg(g):(\C\times S^1)\backslash B\to\C$.

\begin{definition}\label{def:mult}
Let $B$ be a geometric braid on $s$ strands, parametrized as in Eq.~\eqref{eq:braidpara}, but with the property that $u_j(t)\neq 0$ for all $t\in[0,2\pi]$. That is, none of the strands intersects the line $\{0\}\times[0,2\pi]\subset\C\times[0,2\pi]$. Let $g(\cdot,\rme^{\rmi t})$ be the corresponding loop of polynomials given by $g(u,\rme^{\rmi t})=\prod_{j=1}^s(u-u_j(t))$. Then we say that $B$ is \textbf{P-fibered with $O$-multiplicity $m$} if $\arg(u^m g):(\C\times S^1)\backslash(B\cup(\{0\}\times S^1))\to S^1$ is a fibration map, i.e., if it has no critical points.
\end{definition}


P-fibered braids of $O$-multiplicity 0 are P-fibered braids. If $B$ is a P-fibered braid of $O$-multiplicity 1, then $B\cup(\{0\}\times[0,2\pi])$ is a P-fibered braid. If the $O$-multiplicity $m$ is greater than 1, the closure of $B\cup(\{0\}\times[0,2\pi])$ is not necessarily a fibered link, since in this case $\arg(u^m g)$ does not take the required form on a tubular neighbourhood of $\{0\}\times S^1$.
Note that, just like for P-fibered braids, a geometric braid $B$ is P-fibered with $O$-multiplicity $m$, if and only if $B^n$ is P-fibered with $O$-multiplicity $m$ for all $n\in\mathbb{Z}\backslash\{0\}$.

\begin{teo}\label{cor}
Let $(B_1,B_2,\ldots,B_N)$ be a sequence of geometric braids such that $B_i$ consists of $s_i$ strands and is P-fibered with $O$-multiplicity $m_i:=\sum_{j<i}s_j$ for all $i\in\{1,2,\ldots,N\}$. Let $r_N:=1$.\\ Then for every sequence of positive integers $(r_j,r_{j+1},\ldots,r_{N})$ there is a positive integer $N_{j-1}$ such that the closure of $B(B_1^{2r_1},B_2^{2r_2},\ldots,B_{N-1}^{2r_{N-1}},B_N^{2r_N})$ is real algebraic if $r_j>N_j$ for all $j=1,2,\ldots,N-1$.
\end{teo}
\begin{proof}
Let $g_i(\cdot,\rme^{\rmi t})$ denote the loop in the space of monic polynomials corresponding to the geometric braid $B_i$ and let $\tilde{g}_i(u,\rme^{\rmi t}):=u^{m_i}g_i(u,\rme^{\rmi t})$. Then the zeros of $\tilde{g}_N(u,\rme^{2\rmi t})$ form the union of the geometric braid $B_N^2$ and the 0-strand $\{0\}\times[0,2\pi]$. As in \cite{Bode2019} we define the $u$-convenient, semiholomorphic, radially weighted homogeneous polynomial
\begin{equation}
f_N(u,v,\bar{v})=f_N(u,r\rme^{\rmi t},r\rme^{-\rmi t}):=r^{k_Ns_N}\tilde{g}_N\left(\frac{u}{r^{k_N}},\rme^{2\rmi t}\right)
\end{equation}
with $k_N$ some sufficiently large even integer. Recall that $f_N$ has no critical points in $(\C^*)^2$ if $\arg(\tilde{g}_N)$ is a fibration map.
Since $B_N^2$ is P-fibered with $O$-multiplicity $m_N$, we have a fibration map $\arg(\tilde{g}_N)$, which implies the non-existence of critical points of $f_N$ in $(\C^*)^2$. Let $a_N(r,t)$ denote the coefficient of $u^{m_N}$ in $f_N$. Since the strands of $B_N^2$ do not intersect the 0-strand, $a_N(r,t)$ is never zero if $r>0$. We define
\begin{equation}\label{eq:fn1}
f_{N-1}(u,v,\bar{v})=f_{N-1}(u,r\rme^{\rmi t},r\rme^{-\rmi t}):=a_N(r,t)r^{k_{N-1}s_{N-1}}\tilde{g}_{N-1}\left(\frac{u}{r^{k_{N-1}}},\rme^{2r_{N-1}\rmi t}\right),
\end{equation} 
where $k_{N-1}$ is some large even integer, greater than $k_N$, and for now $r_{N-1}$ is an arbitrary integer. Note that the geometric braid formed by the zeros of $\tilde{g}_{N-1}(u,\rme^{2r_{N-1}}\rmi t)$ is the union of $B_{N-1}^{2r_{N-1}}$ and the 0-strand. We want to show that if $r_{N-1}$ is sufficiently large, then $f_{N-1}$ has no critical points in $(\C^*)^2$. Since $f_{N-1}$ is radially weighted homogeneous and semiholomorphic, this is true if $\arg(a_N(1,t)\tilde{g}_{N-1}(u,\rme^{2r_{N-1}\rmi t}))$ is a fibration map. 

Let $v_{N-1,j}(t)$, $j=1,2,\ldots,s_{N-1}-m_{N-1}-1$, denote the non-zero critical values of $\tilde{g}_{N-1}(\cdot,\rme^{\rmi t})$. Since $B_{N-1}$ is P-fibered with $O$-multiplicity $m_{N-1}$, we have a fibration map $\arg(\tilde{g}_{N-1})$, which is equivalent to $\tfrac{\partial \arg(v_{N-1,j})}{\partial t}(t)\neq 0$ for all $j=1,2,\ldots,s_{N-1}-m_{N-1}-1$ and all $t\in[0,2\pi]$. Then the non-zero critical values of $a_N(1,t)\tilde{g}_{N-1}(u,\rme^{2r_{N-1}\rmi t})$ are $a_N(1,t)v_{N-1,j}(2r_{N-1}t)$ and hence
\begin{equation}
\frac{\partial \arg(a_N(1,t)v_{N-1,j}(2r_{N-1}t))}{\partial t}=\frac{\partial \arg(a_N(1,t))}{\partial t}+2r_{N-1}\frac{\partial \arg(v_{N-1,j})}{\partial t},
\end{equation}
which is non-vanishing as long as $r_{N-1}$ is chosen sufficiently large. Therefore, $\arg(a_N(1,t)\tilde{g}_{N-1}(u,\rme^{2r_{N-1}\rmi t}))$ is a fibration map and $f_{N-1}$ has no critical points in $(\C^*)^2$ if $r_{N-1}$ is large. Note that since we introduced the factor $a_N$ in Eq.~\eqref{eq:fn1}, the coefficient of $u^{s_{N-1}+m_{N-1}}$ (the highest order term) in $f_{N-1}$ is the same as the coefficient of $u^{m_N}=u^{s_{N-1}+m_{N-1}}$ (the lowest order term) in $f_N$. 

We proceed like this inductively, defining $a_{j}(r,t)$ to be the coefficient of $u^{m_j}$ in $f_j$ and
\begin{equation}
f_{j-1}(u,v,\bar{v})=f_{j-1}(u,r\rme^{\rmi t},r\rme^{-\rmi t}):=a_j(r,t)r^{k_{j-1}s_{j-1}}\tilde{g}_{j-1}\left(\frac{u}{r^{k_{j-1}}},\rme^{2r_{j-1}\rmi t}\right),
\end{equation}
where $k_{j-1}$ is chosen greater than $k_j$. As above, we can guarantee that $f_{j-1}$ has no critical points in $(\C^*)^2$ by choosing $r_{j-1}$ sufficiently large. Note however that $a_j(r,t)$ depends on $r_j$ and $k_j$ and the range of values for $r_{j-1}$ that is sufficient depends on $a_j(r,t)$. Likewise, the range of values of $k_j$ that is sufficient depends on the choice of $r_j$. Therefore, we cannot know how to chose $r_{j-1}$ before we have chosen $r_j$ and then $k_j$.

There is one case for which the proof that $f_1$ is strongly non-degenerate differs slightly from the above. Namely, if $s_1=1$, then there are no critical values of $\tilde{g}_1(\cdot,\rme^{\rmi t})$ for any $t\in[0,2\pi]$. In other words, $\tfrac{\partial f_1}{\partial u}$ never vanishes, so that $f_1$ is strongly non-degenerate. Note that in this case there is no restriction on our choice of $r_1\neq 0$.

Since the lowest order term of $f_j$ (with respect to $u$) is equal to the highest order term of $f_{j-1}$, we can combine the different $f_j$s to one semiholomorphic boundary polynomial $f$ with face functions $f_{P_i}=f_i$, where $P_i=(k_i,1)$.

We have already shown that the face functions $f_j$ associated with 1-faces are strongly non-degenerate. Since $f$ is semiholomorphic, all non-extreme vertices are associated with strongly non-degenerate face functions as well. Therefore $f$ satisfies Condition (ii) in Definition~\ref{def-strong}. 

Thus in order to prove that $f$ is strongly inner non-degenerate we only need to check that $f_{P_1}=f_1$ has no critical points in $\C^2\backslash\{(u,0):u\in\C\}$ and $f_{P_N}=f_N$ has no critical points in $\C^2\backslash\{(0,v):v\in\C\}$. Since we already know that $f_1$ and $f_N$ have no critical points in $(\C^*)^2$ we only need to check points in $\{(0,v):v\in\C^*\}$ and $\{(u,0):u\in\C^*\}$, respectively.

By construction the only term of $f_N$ that does not depend on $v$ or $\bar{v}$ is $u^{\sum_{i=1}^N s_i}$, so that $(f_N)_u(u,0,0)=\left(\sum_{i=1}^N s_i\right) u^{\left(\sum_{i=1}^N s_i\right)-1}$, whose only root is $u=0$. Hence $f_N$ satisfies Condition (i) in Definition~\ref{def-strong}.

Since by assumption $B_1$ has $O$-multiplicity $m_1=0$, we have that $f_1$ is $v$-convenient.

Recall that, since $f_1$ is radially weighted homogeneous, semiholomorphic and $v$-convenient, its critical points in $\C\times\C^*$ correspond to critical points of $\arg(a_2(1,t)\tilde{g}_1(u,\rme^{2r_1\rmi t}))$. However, $r_1$ is chosen such that $\arg(a_2(1,t)\tilde{g}_1(u,\rme^{2r_1\rmi t}))$ is a fibration, so $f_1$ satisfies Condition (i) in Definition~\ref{def-strong} and $f$ is strongly inner non-degenerate.

 
It follows from Proposition~\ref{strong-isolated} that $f$ has an isolated singularity. Since strong inner non-degeneracy implies inner non-degeneracy, we know from Proposition~\ref{prop:main} that the link of its singularity is given by the closure of the braid $B(B_1^{2r_1},B_2^{2r_2},\ldots,B_N^{2r_N})$.
\end{proof}

\bibliographystyle{amsplain}
  \renewcommand{\refname}{ \large R\normalsize  EFERENCES}
\bibliography{sample}
\Addresses
\end{document}